\documentclass[a4paper, 11pt,reqno]{amsart}
\usepackage[top = 1in, bottom = 1in, left=1.1in, right=1.1in]{geometry}

\usepackage[utf8]{inputenc}
\usepackage[T1]{fontenc}

\usepackage{amssymb,amsmath,amsthm,graphicx,xcolor,mathtools,mathrsfs}

\usepackage{hyperref}
\hypersetup{colorlinks, linkcolor={red!50!black}, citecolor={green!50!black}, urlcolor={blue!50!black}}

\usepackage[
backend=biber,
natbib=true,
style=alphabetic,
isbn=false,
maxbibnames=9
]{biblatex}
\usepackage{csquotes}
\renewbibmacro{in:}{}
\DeclareFieldFormat{pages}{#1}

\DeclareFieldFormat{title}{\mkbibquote{#1}}
\DeclareFieldFormat[misc]{date}{Preprint (#1)}

\renewbibmacro*{doi+eprint+url}{
	\iftoggle{bbx:url}
	{\iffieldundef{doi}{\usebibmacro{url+urldate}}{}}{}
	\newunit\newblock
	\iftoggle{bbx:eprint}
	{\usebibmacro{eprint}}
	{}
	\newunit\newblock
	\iftoggle{bbx:doi}
	{\printfield{doi}}
	{}
}

\usepackage{cleveref}
\theoremstyle{plain}
\newtheorem{theorem}{Theorem}[section]
\crefname{theorem}{Theorem}{Theorems}

\newtheorem{proposition}[theorem]{Proposition}
\crefname{proposition}{Proposition}{Propositions}

\newtheorem{corollary}[theorem]{Corollary}
\crefname{corollary}{Corollary}{Corollaries}

\newtheorem{lemma}[theorem]{Lemma}
\crefname{lemma}{Lemma}{Lemmas}

\crefname{conjecture}{Conjecture}{Conjectures}

\crefname{problem}{Problem}{Problem}

\newtheorem{claim}[theorem]{Claim}
\crefname{claim}{Claim}{Claims}

\crefname{observation}{Observation}{Observations}

\crefname{setup}{Setup}{Setups}

\crefname{myth}{Myth}{Myths}

\newtheorem{fact}[theorem]{Fact}
\crefname{fact}{Fact}{Facts}

\crefname{algorithm}{Algorithm}{Algorithms}

\crefname{remark}{Remark}{Remarks}

\crefname{example}{Example}{Examples}

\theoremstyle{definition}
\newtheorem{definition}[theorem]{Definition}
\crefname{definition}{Definition}{Definitions}

\crefname{construction}{Construction}{Constructions}

\crefname{question}{Question}{Questions}

\numberwithin{equation}{section}

\usepackage[shortlabels]{enumitem}
\setlist[enumerate,1]{label={\upshape (\roman*)}}

\DeclareMathOperator{\probability}{Pr}
\DeclareMathOperator{\expectation}{\mathbf{E}}

\allowdisplaybreaks

\newenvironment{proofclaim}[1][Proof of the claim]{\begin{proof}[#1]}{\end{proof}}

\author[P.~Ara\'ujo]{Pedro Ara\'ujo}
\address[P.~Ara\'ujo]{Department of Mathematics, Faculty of Nuclear Sciences and Physical Engineering, Czech Technical University in Prague, Trojanova 13, 120 00, Prague, Czech Republic}
\email{pedro.araujo@cvut.cz}

\author[M.~Pavez-Sign\'e]{Mat\'ias Pavez-Sign\'e}
\address[M.~Pavez-Sign\'e]{Centro de Modelamiento Matem\'atico (CNRS IRL2807), Universidad de Chile, Santiago, Chile.}
\email{mpavez@dim.uchile.cl}

\author[N.~Sanhueza-Matamala]{Nicol\'as Sanhueza-Matamala}
\address[N.~Sanhueza-Matamala]{Departamento de Ingeniería Matemática, Facultad de Ciencias Físicas y Matemáticas, Universidad de Concepción, Chile.}
\email{nicolas@sanhueza.net}

\thanks{PA was supported by the long-term strategic development financing of the Institute of Computer Science (RVO: 67985807) and by the Czech Science Foundation, grant number GJ20-27757Y. NSM was partly supported by the Czech Science Foundation, grant number GA19-08740S with institutional support RVO: 67985807, while he was affiliated with the Institute of Computer Sciences of the Czech Academy of Sciences; and also by ANID-Chile through the FONDECYT Iniciación Nº11220269 grant. MPS was supported by ANID Basal grant CMM FB210005, by ANID Regular grant N°1241398, and by the European
Research Council (ERC) under the European Union Horizon 2020 research and innovation programme (grant agreement No. 947978) while the author was affiliated with the University of Warwick.}

\title{Ramsey numbers of cycles in random graphs}

\begin{document}

\begin{abstract}
Let $R(C_n)$ be the Ramsey number of the cycle on $n$ vertices.
We prove that, for some $C > 0$, with high probability every $2$-colouring of the edges of $G(N,p)$ has a monochromatic copy of $C_n$, as long as $N\geq R(C_n) + C/p$ and $p \geq C/n$.
This is sharp up to the value of $C$ and it improves results of Letzter and of Krivelevich, Kronenberg and Mond. 
\end{abstract}

\maketitle

\section{Introduction}
 The \emph{Ramsey number} of a graph $H$, denoted by $R(H)$, is the minimum number $N$ so that every $2$-colouring of the edges of the complete graph on $N$ vertices yields a monochromatic copy of $H$. Here, we focus in the case where $H$ is an $n$-vertex cycle $C_n$, in which case it holds~\cite{BondyErdos1973, Rosta1973, FaudreeSchelp1974} that $R(C_3) = R(C_4) = 6$, and
\begin{equation*}
    R(C_n) = \begin{cases}
    2n-1 & \text{if $n \geq 5$ is odd,} \\
    3n/2-1 & \text{if $n \geq 6$ is even.}
    \end{cases}
\end{equation*}
A way to extend this further is by trying to find monochromatic copies of $C_n$ in $2$-edge-coloured host graphs on $R(C_n)$ vertices (or close to $R(C_n)$) which are not complete. In this paper, we investigate the question of whether a random graph is Ramsey for the cycle.

The \emph{binomial random graph} $G(n,p)$ has $n$ vertices and each possible edge appears independently with probability $p = p(n)$.
A sequence of events, indexed by $n\in\mathbb N$, happens \emph{with high probability} (w.h.p.) if the probability of those events tends to $1$ whenever $n$ tends to infinity.
For graphs $G$ and $H$, we say that $G$ is \emph{$H$-Ramsey}, denoted by $G\rightarrow H$, if every $2$-colouring of the edges of $G$ yields a monochromatic copy of $H$.
Here, the relevant question is, for which values of $p$, and $N$ ``close'' to $R(C_n)$, does it hold that $G(N,p) \rightarrow C_n$ w.h.p.?

Recently, Krivelevich, Kronenberg, and Mond~\cite{KKM2020} obtained Ramsey-type results for cycles in 2-edge-colourings of random graphs (extending analogous results of Letzter~\cite{L2016} for paths).
Their results for odd and even cycles can be compactly stated as follows:
for every $\varepsilon > 0$, there exists $C > 0$ such that if $p \geq C/n$, then w.h.p. $G(R(C_n) + \varepsilon n, p) \rightarrow C_n$.

Those results consider random graphs on $R(C_n) + \varepsilon n$ vertices, instead of just $R(C_n)$.
The case $p = \Theta(1/n)$ shows that the extra $\varepsilon n$ vertices are necessary (we will say more about this later).
But what can we say for other values of $p$?
Balogh et al.
~\cite{BKLL2022} showed that for any large even $n$, every graph $G$ on $N = R(C_{n})$ vertices with $\delta(G) \geq 3N/4$ is $C_n$-Ramsey.
Thus, for constant $p > 3/4$,
we see that w.h.p. $G(R(C_{n}),p) \rightarrow C_{n}$, so here no extra vertices are necessary.
In a similar fashion, a recent result of \L{}uczak, Polcyn and Rahimi~\cite{LPR2022} implies that for odd $n$ and constant $p > 1/2$, $G(R(C_{n}),p) \rightarrow C_{n}$.

Our main result bridges the gap between these two extreme situations ($p = \Theta(1/n)$ and $p = \Theta(1)$), obtaining the essentially best-possible Ramsey-type result for cycles in random host graphs, and every possible value of $p$.

\begin{theorem}
	\label{theorem:main}
	There exist positive constants $c$ and $C$ such that w.h.p.
	\begin{enumerate}
		\item \label{item:maincycles} $G(N, p) \rightarrow C_n$ if $N\geq R(C_n)+C/p$ and $p\geq C/n$,
		\item \label{item:bestpossible} $G(N, p) \nrightarrow C_n$ if $N\leq R(C_n)+c/p$ and $p \leq c$, or if $p\leq c/n$.
	\end{enumerate}
\end{theorem}
We can show part \ref{item:bestpossible} right away.
Note this means that part~\ref{item:maincycles} is best possible up to the specific values of $c$ and $C$.
\begin{proof}[Proof of Theorem~\ref{theorem:main}\ref{item:bestpossible}]
    For $p \leq 1/n$, it is well-known that the size of the largest connected component of $G(n,p)$ is $O(n^{2/3})$, and thus $G(n,p)$ has no cycle on $\Omega(n)$ vertices.
    Thus we can assume from now on that $1/n \leq p \leq c$.
    
    Suppose then that $N\le R(C_n)+c/p$ and $1/n \leq p \leq c$.
    Note that the expected number of vertices adjacent to a set of size $c/p + 2$ is at most $Np(c/p + 2) \leq 9cn$ (we used in the last inequality that $N \leq 2n + c/p \leq 2n + cn \leq 3n$ and $p \leq c$).
    Using this, it can be checked that w.h.p. there exists a set $X$ of size $c/p+2$ such that the size of the neighbourhood of $X$ is at most $10cn \leq n/2 - 1$.
    Note also that $|X| \leq cn + 2 \leq n-1$.
    
    We consider first the case where $n$ is odd.
    In this case, $R(C_n) = 2n-1$.
    Let $G=G(N,p)$, where $N=2n+c/p$.
    We partition $V(G)=V_1\cup V_2\cup X$ so that $|X| = c/p + 2$, $N(X)\subset V_1$ and $|V_1|=|V_2|=n-1$
    (this is possible since $|N(X)| \leq 10cn \leq n/2 - 1 < n-1$).
    We colour the edges of $G$ in such a way that all edges inside $V_1$, $V_2$ and $X$ are red, and all remaining edges are blue. 
    This colouring cannot contain a monochromatic cycle of length $n$ as the blue components are bipartite, and each red component has at most $n-1$ vertices.
    
    Now we consider even $n$.
    If $n$ is even, then $N=R(C_n)+c/p = 3n/2+c/p$.
    We partition $V(G)=V_1\cup V_2\cup X$ so that $|V_1|=n-1$, $|V_2|=n/2-1$, $|X| = c/p + 2$ and $N(X)\subset V_2$ (again, this is possible since $|N(X)| \leq n/2 - 1$).
    We colour all edges inside $V_1$ with red, all edges between $V_1\cup X$ and $V_2$ in blue, all edges inside $X$ in red, and use any colour for the edges inside $V_2$.
    No red cycle of length $n$ can exist as the red components have size at most $n-1$,
    and there is no blue cycle on $n$ vertices in $G$ as otherwise it should use at least $n/2>|V_2|$ vertices from $V_2$.
\end{proof}

\subsection{Proof sketch} \label{section:proofsketch}
Our proof of Theorem~\ref{theorem:main}\ref{item:maincycles} proceeds in three steps.
Suppose that $n$ is given, $N = R(C_n) + C/p$, and we have a red-blue edge-colouring of $G = G(N, p)$ without any monochromatic $n$-cycle.

\begin{enumerate}
    \item \emph{Step 1: Coarse stability.}
    We show first that the edge-colouring of $G$ must in fact follow closely a global pattern, which essentially resembles the colourings of cliques on $R(C_n)-1$ vertices without any monochromatic $C_n$.
    Up to renaming the colours, the structure can be described as follows.
    We find a partition of $V(G)$ into two sets, $V_1$ and $V_2$, such that the edges between $V_1$ and $V_2$ are almost all blue, and almost all edges inside $V_1$ are red.
    If $n$ is odd, then we also know that $|V_1|, |V_2|$ are very close to $n$ and that almost all edges inside $V_2$ are red; if $n$ is even, then we know instead that $|V_1|$ and $2|V_2|$ are close to $n$.
    
    \item \emph{Step 2: Fine stability.}
    Here, we refine the structure found in the first step.
    By reallocating vertices carefully, we will transform the partition $V_1, V_2$ to a partition of $V(G)$ into three sets, $V'_1, V'_2$ and $W$. Here, $V'_1$ and $V'_2$ behave similarly in terms of the colouring of $V_1, V_2$, but now we can also ensure that every vertex in $V'_1 \cup V'_2$ is adjacent to at least $pn/100$ vertices of the `correct colour' into each cluster (e.g. each vertex in $V'_1$ has $pn/100$ blue neighbours in $V'_2$ and $pn/100$ red neighbours in $V'_1$).
    Moreover, and crucially, we can also ensure that $W$ is very small, of size at most $C/p$.
    
    \item \emph{Step 3: Cycle finding.}
    Here, depending on the relative sizes of $V'_1$ and $V'_2$, we will find a cycle inside one of the monochromatic subgraphs induced by $V'_1$, by $V'_2$, or in the bipartite subgraph induced by $V'_1$ and $V'_2$.
    For instance, if $V'_1$ has size at least $n$, we will find a red monochromatic cycle of length $n$ inside $G[V'_1]$.
    Since $|W| \leq C/p$, there will always be one choice that works, and thus we finish the proof.
\end{enumerate}

For the proof of Step 1, we will use the technique of (multicoloured) sparse regularity, combined with known stability results for cycle-free colourings in almost complete graphs.
Similar techniques were used by Letzer~\cite{L2016} and Krivelevich, Kronenberg and Mond~\cite{KKM2020} to find long cycles and paths in edge-coloured random graphs.
In fact, this first step is enough if we were interested in finding monochromatic cycles or paths of length $(1 - o(1))n$ instead; the next two steps are crucial to improve this to length exactly $n$.

For the proof of Step 2, we refine the structure found in Step 1 in a vertex-by-vertex fashion, and we use \emph{tree embeddings} in expanders and bipartite expanders as a main tool.

For the proof of Step 3, we use the rotation-extension technique in the monochromatic subgraphs induced by $G[V'_1]$ and $G[V'_2]$.
Here we rely heavily on the expansion properties of subgraphs of random graphs which also have minimum degree conditions.
We also need to tailor this method to work in bipartite expander subgraphs.

\subsection{Organisation}
The paper is organised as follows.
In Section~\ref{section:preliminaries} we fix notation and then gather known results about expansion, random graphs, and sparse regularity.
In Section~\ref{section:stability} we prove our `coarse stability' result (Theorem~\ref{theorem:stability:1}), which corresponds to Step 1 in the sketch above.
In the next two sections, we provide the tools to proceed with the `cycle finding' of Step 3, the main result here is the Cycle finder lemma (Lemma~\ref{lemma:cyclesprescribedlength}).
In Section~\ref{section:rotation} we use the rotation-extension method in expanders; in Section~\ref{section:cycles} we argue that the monochromatic subgraphs of random graphs are indeed expanders.
Finally, the pieces to prove Theorem~\ref{theorem:main} are put together in Section~\ref{section:mainproof}.
Some lemmas about tree embeddings in bipartite expanders are defered to Appendix~\ref{appendix:bipartiteextendibility}.

\section{Preliminaries} \label{section:preliminaries}

Our notation here is standard.
Given a graph $G$, we write $V(G)$ and $E(G)$ for its vertex set and edge set, respectively.
We denote by $|G|=|V(G)|$ its number of vertices and write $e(G)=|E(G)|$ for its number of edges.
For a subset $U\subset V(G)$, $G[U]$ is the graph induced on $U$ and set $e(U)=e(G[U])$.
Given disjoint subsets $A,B\subset V(G)$, we write $G[A,B]$ for bipartite graph induced by $A$ and $B$, that is, the graph with vertex set $A\cup B$ and all the edges of the form $e=ab\in E(G)$ with $a\in A$ and $b\in B$, and set $e(A,B)=e(G[A,B])$. Given a vertex $v\in V(G)$, the neighbourhood of $v$ is denoted by $N(v)$ and its degree is $d(v)=|N(v)|$.
The maximum degree of $G$ is denoted by $\Delta(G)$, and the minimum degree by $\delta(G)$.
For a subset $U\subset V(G)$, we write $N(v,U)$ for the set of neighbours of $v$ in $U$.
The set of neighbours of $U\subset V(G)$ is $\Gamma(U)=\bigcup_{u\in U}N(u)$ and its \textit{external neighbourhood} is $N(U)=\bigcup_{u\in U}N(u)\setminus U$.
Furthermore, we write $N(U,W)=N(U)\cap W$.
When working with more than one graph, we indicate the graph considered in the subscript in order to avoid confusion, for example, $e_G(A,B)$ indicates the number of edges between $A$ and $B$ in $G$.

\subsection{Expansion}
Informally speaking, a graph $G$ is an expander if every sufficiently small set of vertices $U\subset V(G)$ has many neighbours outside $U$, in which case we say that $U$ \textit{expands} into $V(G)\setminus U$.

\begin{definition}[Expander graph]Let $G$ be a graph,  let $m,m'\in\mathbb N$, and let $d>0$.
\begin{enumerate}
    \item $G$ is an \emph{$(m, d)$-expander} if for every $S \subset V(G)$ with $|S| \leq m$, we have $|N(S)| \geq d|S|$.
    \item $G$ is \emph{$(m, m')$-joined} if every pair of disjoint sets $A, B \subset V(G)$, with $|A| = m$ and $|B| = m'$, have an edge between them.
    If $m=m'$, then we say that $G$ is \emph{$m$-joined}.
\end{enumerate}
\end{definition}
The following observation shows that we can get improved expander parameters in a graph which is both an expander and joined.

\begin{lemma} \label{lemma:improvedexpansion}
    Let $M$, $m$, $n$ satisfy $m \leq M$.
    Suppose $G$ is an $n$-vertex graph which is $(m,M)$-joined.
    Then, for each $X \subseteq V(G)$ with $|X| \geq m$, $|N(X)| \geq n - M +1 - |X|$.
    In particular, if $G$ is also a $(k,d)$-expander and $k+1 \geq m$, then $G$ is also a $((n-M+1)/(d+1), d)$-expander.
\end{lemma}

\begin{proof}
    Let $X$ be a subset of size at least $m$,
    and let $X' \subseteq X$ have size exactly $m$.
    If $|N(X)| \leq n - M - |X|$, then there exists a set $Y \subseteq V(G) \setminus (N(X) \cup X)$ of size at least $M$, which means there are no edges between $X'$ and $Y$ and contradicts that $G$ is $(m, M)$-joined.
    
    Now suppose $G$ is in addition a $(k,d)$-expander and $k + 1 \geq m$.
    To show that $G$ is a $((n-M+1)/(d+1), d)$-expander, let $X \subseteq V(G)$ with $1 \leq |X| \leq (n-M+1)/(d+1)$, we need to show that $|N(X)| \geq d|X|$.
    If $|X| \leq k$ this follows since $G$ is a $(k,d)$-expander.
    Otherwise, we have $|X| \geq k+1 \geq m$, so
    $|N(X)| \geq n - M + 1 - |X| \geq d |X|$, as required.
\end{proof}

The following lemma says that we can remove a few vertices from an expander graph so that it remains expander.

\begin{lemma} \label{lemma:friedmanpippenger-expansion}
    Let $d \geq 2$ and let $G$ be a 
    $(2k-2,d)$-expander.
    For any $0 \leq r \leq k$, there exists a set $X \subseteq V(G)$ of size $r$ such that $G - X$ is a $(2k-2,d-1)$-expander.
\end{lemma}

\begin{proof}
    For $d \geq 3$, this can be seen to be true by inspecting the proof of~\cite[Theorem 1]{FP1987} while embedding an $r$-vertex path.
    For completeness, we give a proof following that outline.
    
    We use induction on $r$.
    If $r = 0$ this is trivial, and if $r = 1$ then it is easy to see we can select an arbitrary vertex as $X$.
    So suppose $0 < r \leq k-1$ and that there exists a set $X$ of size $r$ such that $G - X$ is a $(2k-2, d-1)$-expander.
    Fix such a set $X$ from now on.
    We shall prove the statement for $r+1$, i.e. prove the existence of a set $X'$ of size $r+1$ such that $G - X'$ is a $(2k-2, d-1)$-expander.
    
    Say a set $Y \subseteq G - X$ is \emph{critical} if $|Y| \leq 2k-2$ and $|N_{G-X}(Y)| = (d-1)|Y|$.
    Note that a critical set can have size at most $k-1$.
    Indeed, since $G$ is a $(2k-2, d)$-expander, then $(d-1)|Y| = |N_{G-X}(Y)| \geq |N_G(Y)| - |X| \geq d|Y| - |X| \geq d|Y| - k + 1$, and the bound follows.
    On the other hand, the union of critical sets is again critical.
    Indeed, if $Y, Z$ are critical sets, then both have size at most $k-1$  and thus $|Y \cup Z| \leq 2k - 2$, and since $G - X$ is a $(2k-2, d-1)$-expander we have $(d-1)|Y \cap Z| \leq |N_{G-X}(Y \cap Z)|$ and also $(d-1)|Y \cup Z| \leq |N_{G-X}(Y \cup Z)| \leq |N_{G-X}(Y)|+|N_{G-X}(Z)|-|N_{G-X}(Y \cap Z)| \leq (d-1)(|Y|+|Z|-|Y \cap Z|) = (d-1)|Y \cup Z|$, thus indeed $Y \cup Z$ is critical.
    
    If for some $v \in V(G) - X$ we have that $G - X - v$ is a $(2k-2,d-1)$-expander, we are done.
    So we can assume that for every $v \in V(G) - X$ there exists a set $Y_v \subseteq V(G) - X - v$ of size at most $2k-2$ such that $|N_{G-X-v}(Y_v)| < (d-1)|Y_v|$.
    Since $G - X$ is a $(2k-2, d-1)$-expander, it must hold that $|N_{G-X}(Y_v)| = (d-1)|Y_v|$ and $v \in N_{G-X}(Y_v)$.
    Therefore each $Y_v$ is critical, and thus $|Y_v| \leq k-1$ for each $v \in V(G) - X$, and $Y^\ast = \bigcup_{v \in V(G) - X} Y_v$ is also critical, and therefore $|N_{G-X}(Y^\ast)| = (d-1)|Y^\ast|$.
    But, since $v \in N_{G-X}(Y_v)$ for each $v \in V(G) - X$, we deduce $V(G) - X - Y^\ast \subseteq N_{G-X}(Y^\ast)$.
    Using that $G$ is a $(2k-2,d)$-expander and $|X \cup Y^\ast| \leq 2k-2$, we have $d|Y^\ast| \leq d|Y^\ast \cup X| \leq |N_{G}(X \cup Y^\ast)| \leq |V(G)| - |X| - |Y^\ast| = |N_{G-X}(Y^\ast)| = (d-1)|Y^\ast|$, a contradiction.
\end{proof}
The following lemma is similar in spirit to Haxell's tree embedding result~\cite{H2001} and can be found in~\cite[Corollary 3.7]{M2019}.
\begin{lemma}\label{lemma:embedding:expander}Let $n,m,d\in\mathbb N$ satisfy $d\ge 4$. Let $G$ be a $(2m,d)$-expander graph on $n$ vertices which is also $m$-joined. Let $T$ be a tree with maximum degree at most $d/2$ such that $|T|\le n-(2d+3)m$. Then, for every vertex $t\in V(T)$ and $v\in V(G)$, there exists a copy of $T$ into $G$ such that $t$ is mapped into $v$. 
\end{lemma}
\subsection{Bipartite expansion}
It is also possible, and useful, to define an expansion-like property in bipartite graphs.

\begin{definition}[Bipartite expander graph]Let $G$ be a bipartite graph with parts $V_1$ and $V_2$.
Let $m\in\mathbb N$, and let $d>0$.
\begin{enumerate}
    \item $G$ is an \emph{$(m, d)$-bipartite-expander} if for each $i \in \{1, 2\}$, and for every $S \subset V_i$ with $|S| \leq m$, we have $|N(S)| \geq d|S|$.
    \item $G$ is \emph{$m$-bipartite-joined} if for every pair of sets $A \subseteq V_1$ and $B \subset V_2$, with $|A| = |B| = m$, we have $e(A,B)>0$.
\end{enumerate}
\end{definition}
The following is a bipartite version of Lemma~\ref{lemma:friedmanpippenger-expansion}.

\begin{lemma} \label{lemma:friedmanpippenger-expansion-bipartite}
    Let $d \geq 3$ and let $G$ be a bipartite graph on classes $V_1, V_2$ which is a $(2k,2d+5)$-bipartite-expander.
    For any $0 \leq r \leq k$, there exists a set $X \subseteq V_2$ of size $r$ such that $G - X$ is a $(2k,d)$-expander.
\end{lemma}

The proof of Lemma~\ref{lemma:friedmanpippenger-expansion-bipartite} can be done similarly as the proof of Lemma~\ref{lemma:friedmanpippenger-expansion} after adapting the concept of expanders and joined to bipartite graphs.
The proof is given in Appendix~\ref{section:appendix-proofs}.

\subsection{Random graphs}

Here we collect some properties that are satisfied, with high probability, by random graphs.

\begin{definition} For $p,\eta\in (0,1]$, we say that a graph $G$ is $(\eta,p)$\textit{-uniform} if for all disjoint sets $A,B\subset V(G)$, with $|A|,|B|\geq \eta |G|$, we have
		\begin{equation}\label{eq:uniform}e(A,B)= (1\pm\eta)p|A||B|\quad\text{and}\quad e(A)= (1\pm\eta)p\dbinom{|A|}{2}.\end{equation}
		Furthermore, if only the upper bounds  in~\eqref{eq:uniform} are known to hold, then we say that $G$ is \emph{$(\eta,p)$-upper-uniform}.
\end{definition}

We also need a similar property that also says something about small sets $A$.

\begin{definition} For $p,\eta\in (0,1]$, we say that a graph $G$ is $(\eta,p)$\emph{-upper-small-uniform} if for all disjoint sets $A,B\subset V(G)$ with $|A| < \eta n$, we have
\begin{equation}
    e(A,B) \leq 7 \eta p n \max \{ \eta n, |B| \} \quad\text{and}\quad e(A) \leq 7 \eta p n^2.
\end{equation}
\end{definition}

The following three results show that random graphs are typically uniform and joined.
They are all straightforward consequences of Chernoff's bound.

\begin{lemma}\label{lemma:upper-uniform}
	For every $\eta>0$ there exists a constant $C>0$ such that if $p\geq C/N$, then w.h.p. $G(N,p)$ is $(\eta,p)$-uniform.
\end{lemma}

\begin{lemma} \label{lemma:upper-small-uniform}
    For sufficiently small $\eta>0$, there exists a constant $C>0$ such that if $p\ge C/N$, then w.h.p. $G(N,p)$ is $(\eta, p)$-upper-small-uniform.
\end{lemma}

\begin{lemma}\label{lemma:gnp:1}
For every $\gamma>0$ and $p\in (0,1)$, w.h.p. $G=G(N,p)$ satisfies the following property. For every pair of disjoint subsets $U,X\subset V(G)$, with $|U|\ge \gamma N$ and $|X|\ge 50/\gamma p$, the number of edges between $U$ and $X$ satisfies $\frac{p}{2}|U||X|\le e(U,X)\le 2p|U||X|$.
\end{lemma}

\begin{lemma} \label{lemma:Gnpisjoined}
The following hold:
\begin{enumerate}
    \item There exists a constant $C$ such that if $p \ge C/n$, then w.h.p. $G(n,p)$ is $(5\log(np)/p)$-joined.
    \item For every $c > 0$ there exists $C > 0$ such that w.h.p. $G(n,p)$ is $(C/p, cn)$-joined.
\end{enumerate}
\end{lemma}

\subsection{Sparse regularity}
	 Let $G$ be a graph and let $p\in(0,1]$.
	 For disjoint sets $A,B\subset V(G)$, the \emph{$p$-density} of the pair $(A,B)$, denoted $d_p(A,B)$, is defined by
	\begin{equation*}d_p(A,B)=\frac{e(A,B)}{p|A||B|}.\end{equation*}
Given $\varepsilon >0$, we say that the pair $(A,B)$ is $(\varepsilon ,p)$-regular (in a graph $G$) if for all subsets $A'\subset A$ and $B'\subset B$, with $|A'|\ge \varepsilon|A|$ and $|B'|\ge \varepsilon|B|$, we have
	\[e(A',B')=(d_p(A,B)\pm\varepsilon)p|A'||B'|.\]
We will use the following standard result about regular pairs.
\begin{lemma}\label{lemma:regularpairs}
Let $0<\varepsilon<\alpha$. If $(A,B)$ is an $(\varepsilon,p)$-regular pair with $p$-density $d>0$, then the following hold.
\begin{enumerate}
			\item For any $A'\subset A$  and $B'\subset B$, with $|A'|\geq \alpha |A|$ and $|B'|\geq \alpha |B|$, the pair $(A',B')$ is $(\varepsilon/\alpha,p)$-regular with $p$-density at least $d-\varepsilon$.
			\item There are at most $\varepsilon |A|$ vertices $v\in A$ such that $d(v,B)<(d-\varepsilon)p|B|$.
\end{enumerate}
\end{lemma}

The Sparse regularity lemma, due to Kohayakawa and R\"odl~\cite{rodlregularity},
states that every upper-uniform graph admits a partition into clusters where most pairs are regular.
We will use the following `colourful' variant which is well-suited for Ramsey-type problems.

\begin{theorem}[Colourful sparse regularity lemma]\label{theorem:regularity}
Given $\varepsilon>0$ and $k_0\in \mathbb{N}$, there are $\eta>0$ and $K_0\ge k_0$ such that the following holds.
For $p\in (0,1)$, let $r\ge 1$ and let $G_1,\dots, G_r$ be $(\eta,p)$-upper-uniform graphs on a common vertex set $V$ of size $n\ge k_0$.
Then there exists a partition $V=V_0\cup V_1\cup \dots \cup V_k$, with $k_0\le k\le K_0$, such that 
\begin{enumerate}
	\item $|V_0|\leq \varepsilon n$,
	\item $|V_i|=|V_j|$ for all $i,j\in[k]$, and
	\item all but at most $\varepsilon \binom k2$ pairs $(V_i,V_j)$ are $(\varepsilon,p)$-regular in $G_t$ for each $t\in [r]$.
\end{enumerate}\end{theorem}
The partition $V=V_0\cup V_1 \cup \dotsb \cup V_k$ given by Theorem~\ref{theorem:regularity} is called an \emph{$(\varepsilon,p)$-regular partition}.
Suppose that a graph $G$ admits an $(\varepsilon,p)$-regular partition $V(G)=V_0\cup V_1\cup\dots \cup V_k$. For $d\in (0,1]$, the \emph{$(\varepsilon,p,d)$-reduced graph} $\Gamma$, with respect to this $(\varepsilon,p)$-regular partition of $G$, is the graph with vertex set $V(\Gamma)=\{V_i:i\in[k]\}$, called \textit{clusters}, where $V_iV_j$ is an edge if and only if $(V_i,V_j)$ is an $(\varepsilon,p)$-regular pair with $d_p(V_i,V_j)\ge d$.

We finish this section with an embedding result that finds almost spanning trees in the bipartite graph induced by large subsets of regular pairs. This lemma is a combination of Corollary~12 and Lemma~19 in~\cite{BCS2011}, nonetheless, it can be proved using the tools we develop in the Appendix.
	\begin{lemma}\label{lemma:regular-tree-embedding}
	 Let $D\geq 2$ and $\varepsilon>0$ satisfy $\varepsilon< 1/(4D+6)$. Let $(V_1,V_2)$ be an $(\varepsilon,p)$-regular pair in a graph $G$, and suppose that $d_p(V_1,V_2)\ge\varepsilon$. Let us further suppose that $|V_1|=|V_2|=m$ and let $V_i'\subset V_i$ satisfy  $|V_i'|\ge (2D+4)\varepsilon m$, for $i\in\{1,2\}$. Then $G[V_1',V'_2]$ contains a copy of every tree $T$ with maximum degree at most $D$ and colour classes with at most $|V_i'|-(2D+1)\varepsilon m$ vertices, for $i\in\{1,2\}$ respectively. 
	\end{lemma}
	
\section{Stability} \label{section:stability}

Extremal results in graph theory often have very structured solutions, and it is usually the case that graphs which are in some sense close to a solution must look approximately like an actual solution, a phenomenon which is known as `stability'. The aim of this section is to show a stability result for the Ramsey problem for paths or cycles in the random graph.
That is, every $2$-colouring of the edges of the random graph without a monochromatic cycle, of appropriate size, should be close to a specific type of colouring which we call \textit{extremal colouring}.

\begin{definition}[Extremal colourings]
Let $\alpha\in (0,1)$ and let $G$ be a graph on $N$ vertices.
We call a red-blue colouring of $E(G)$ \textit{$\alpha$-even-extremal} if there exists a partition $V(G)=X\cup Y$ such that  
\begin{enumerate}
    \item $|X| \geq (1-\alpha)2N/3$ and $|Y| \geq (1-\alpha)N/3 $,
    \item $e_R(X)\ge (1-\alpha)e(X)$ and $e_B(X,Y)\geq (1-\alpha)e(X,Y),$
\end{enumerate}
\noindent and \textit{$\alpha$-odd-extremal} if there exists a partition $V(G)=X\cup Y$ such that  
\begin{enumerate}
    \item $|X| \geq (1-\alpha)N/2$ and $|Y| \geq (1-\alpha)N/2 $,
    \item $e_R(X)\ge (1-\alpha)e(X)$, $e_R(Y)\ge (1-\alpha)e(Y)$ and $e_B(X,Y)\geq (1-\alpha)e(X,Y).$
\end{enumerate}
\end{definition}
We observe that if $G$ is a complete graph, then the colourings which show the lower bound on the Ramsey number for cycles are $0$-even-extremal and $0$-odd-extremal, for even and odd cycles, respectively. Our aim in this section is to prove the following stability result. 

\begin{theorem}[Stability lemma]\label{theorem:stability:1}
For every $\alpha\in (0,1)$, there exist positive constants $\delta<\alpha$ and $K$ such that for $p\ge K/N$ and $N\ge (1-\delta) R(C_n)$, the following holds.
For $G=G(N,p)$, w.h.p. every $2$-colouring of $E(G)$ which contains no monochromatic $C_n$ is either $\alpha$-even-extremal if $n$ is even or $\alpha$-odd-extremal if $n$ is odd.
\end{theorem}

The proof of Theorem~\ref{theorem:stability:1} relies on the regularity method for random graphs. Given a $2$-edge colouring of $G(N,p)$, we apply the regularity lemma to obtain a reduced graph, which will be a $2$-edge-coloured almost-complete graph.
For odd cycles, we use a stability result due to Jenssen and Skokan~\cite{JS2021} to show that the reduced graph either contains a large monochromatic odd cycle or is extremal.
For even cycles, using a result of Letzer~\cite{L2020}, we deduce that the reduced graph either contains a large monochromatic \textit{connected matching} or it is extremal. In both cases, if the reduced graph is non-extremal, using the monochromatic structure that we found in the reduced graph we can embed an $n$-vertex monochromatic cycle in $G(N,p)$. Otherwise, one can show that the reduced graph being extremal implies that the colouring of the edges of $G(N,p)$ should be extremal as well.

\subsection{Stability in almost complete graphs}
In this subsection, we rely on previously established stability results for the Ramsey problem for cycles.
We start with a lemma that appears, in a way more general form, in the work of Jenssen and Skokan~\cite[Theorem 7.4]{JS2021} on the multicolour Ramsey numbers of odd cycles, and is in perfect shape for our applications.

 \begin{proposition}[Stability: odd case]
    \label{proposition:stability:odd}
    For every $\alpha \in (0,1)$, there exists $\delta>0$ such that the following holds for every sufficiently large odd $n$. Let $G$ be a graph with $N\geq (2-\delta)n$ vertices and at least $(1-\delta)\binom{N}{2}$ edges.
    Then every $2$-colouring of $E(G)$ with no monochromatic copy of $C_n$ is $\alpha$-odd-extremal.
\end{proposition}

For the even case, previous results do not exactly fit our intentions.
For this reason, we have to keep an eye for a monochromatic \textit{connected matching}, which is just a matching whose edges all have the same colour and all its vertices are in the same connected component.
We will say that a connected matching has $k$ vertices to specify the size of the matching instead of the whole component. Crucially, finding a large connected matching in the reduced graph of a given regular partition is enough for embedding even cycles of an appropriate size.

 \begin{proposition}[Stability: even case]
    \label{proposition:stability:even}
    For every $B>3/2$ and $\alpha \in (0,1)$, there exists $\delta>0$ such that the following holds for every sufficiently large $n$. Let $N\in\mathbb N$ satisfy $ (3/2-\delta)n\leq N\leq Bn$, and let $G$ be an $N$-vertex graph with at least $(1-\delta)\binom{N}{2}$ edges. Then every $2$-colouring of $E(G)$ either contains a monochromatic connected matching with at least $n$ vertices or is $\alpha$-even-extremal.
\end{proposition}
We will deduce Proposition~\ref{proposition:stability:even} from a stability result for the Ramsey problem for paths, proved by Gy\'arf\'as, S\'ark\"ozy and Szemer\'edi~\cite{GSS2009}.

 \begin{lemma}
    \label{lemma:stab-paths}
    For every $\alpha \in (0,1)$, there exists $\delta>0$ such that the following holds for every sufficiently large $n$. If $N\geq(3/2-\delta)n$, then every $2$-colouring of $E(K_N)$ either contains a monochromatic copy of $P_n$ or is $\alpha$-even-extremal.
\end{lemma}
\cref{lemma:stab-paths} applies only to complete graphs; thankfully, we can apply machinery developed by Letzter~\cite[Theorem 2.1]{L2020} to circumvent this issue.
We slightly simplify its statement for our purposes.
\begin{lemma}
    \label{lemma:almost-to-complete}
       Let $B>1$, let $\varepsilon>0$ be sufficiently small and let $N\leq Bn$. Then there exists a constant $\delta=\delta(\varepsilon,B)>0$ such that the following holds.
       
       Suppose that $G$ is a $2$-coloured graph on at least $N+\varepsilon n$ vertices such that every vertex has at most $\delta n$ non-neighbours, and further suppose that $G$ has no monochromatic connected matching on at least $n$ vertices. Then there exists a $2$-colouring $G'$ of $K_N$ with no monochromatic connected matching on at least $n$ vertices. Moreover, there exists such $G'$ so that it contains an induced subgraph of $G$ on $N$ vertices.
\end{lemma}

With these tools at hand, we are now ready to prove stability for the even case. 

\begin{proof}[Proof of Proposition~\ref{proposition:stability:even}] Let $1/n\ll\delta\ll\delta'\ll \varepsilon'\ll\alpha,1/B$.
Let $N$ satisfy $(3/2 - \delta)n \leq N \leq Bn$ and let $G=(V,E)$ be a $2$-coloured $N$-vertex graph with at least $(1-\delta)\binom{N}{2}$ edges and no monochromatic connected matching on $n$ vertices.
We shall show that the colouring is $\alpha$-even-extremal.

Letting $W\subset V(G)$ be the set of vertices with more than $\delta' n$ non-neighbours in $G$, we get
\[|W|\le \frac{\delta N^2}{\delta 'n}\le \frac{\delta B^2n}{\delta'}\le \varepsilon'n,\]
provided $\delta$ is sufficiently small with respect to $\delta', \varepsilon', 1/B$.
Let $G'=G-W$ and note that every vertex in $G'$ has at most $\delta' n$ non-neighbours.
Let $N'' = |V(G)| - |W| - \varepsilon' n$.
Then, by Lemma~\ref{lemma:almost-to-complete} (applied with $N'', \varepsilon', \delta'$ playing the roles of $N, \varepsilon, \delta)$, we get an induced coloured subgraph $G''\subseteq G'$ on 
$N'' \geq (3/2 - 3 \varepsilon')n$ vertices which is contained in a $2$-colouring $G^*$ of $K_{N''}$ with no monochromatic connected matching on $n$ vertices.
In particular, this implies that $G^*$ contains no monochromatic copy of $P_{n-1}$, and therefore, by Lemma~\ref{lemma:stab-paths}, it is $\alpha/2$-even-extremal.
Since $G''$ is an induced graph of $G'$ and $|V(G)\setminus V(G'')|\le2\varepsilon' n\ll \alpha N$, it is easy to see that $G$ is $\alpha$-even-extremal.
\end{proof}

\subsection{Cycles and connected matchings}

In this subsection, we work with the structure in the reduced graph that will allow us to embed large cycles. Recall that a {connected matching} on $k$ vertices is a $k$-vertex matching such that all its vertices belong to the same connected component. If this connected component is non-bipartite, we call it an \textit{odd connected matching}. This concept was introduced by \L uczak~\cite{Luczak1999} to tackle cycle embedding problems in combination with the Regularity Lemma. Here we combine this approach with tree embedding results due to Balogh, Csaba, and Samotij~\cite{BCS2011} to find cycles in the sparse regime. We remark that this idea has also been used by Kronenberg, Krivelevich, and Mond~\cite{KKM2020} to study the resilience properties of long cycles in sparse random graphs.

We start with a property of odd connected matchings that will play an important role in our argument.

\begin{fact}
\label{fact:good-walk}Suppose that $u$ and $v$ are vertices in an odd connected matching in a graph $H$. Then, there exists a sequence of vertices $u= w_1, w_2,\dots, w_t= v$ such that, for $i\in[t-1]$, $w_iw_{i+1}\in E(H)$ and each vertex in $H$ appears at most three times in the sequence. Moreover, we can find such a walk with the number of vertices being of any given parity. We will refer to such sequence as a \emph{good even-walk} or a \emph{good odd-walk} between $u$ and $v$, depending on its parity.\end{fact}

It is easy to see that Fact~\ref{fact:good-walk} is true. Indeed, for a given pair of vertices $u$ and $v$, we use connectivity to reach a fixed odd cycle by paths from each vertex, from which we may choose a path (possibly a single edge) to control the parity. In this process, we use three paths to build the good walk between $u$ and $v$ of a given parity. 

\begin{definition}Let $H$ be a graph with $V(H)=[h]$, and let $\varepsilon, p \in (0,1)$ and $m\in\mathbb N$. We define the \emph{$(\varepsilon, p,m)$-blow-up of $H$} as the graph constructed as follows. The vertex set consists of pairwise disjoint sets $V_1,\dots, V_h$, with $|V_1|=\dots=|V_h|=m$, and, for every $ij\in E(H)$, we add edges between $V_i$ and $V_j$ so that $(V_i,V_j)$ is an $(\varepsilon,p)$-regular pair with $d_p(V_i,V_j)\ge\varepsilon$.
\end{definition}

The main result of this subsection states that we can find cycles of various lengths in the blow-up of a connected matching.
Krivelevich, Kronenberg, and Mond~\cite{KKM2020} proved a similar result for blow-ups of cycles, and here we generalise their result for blow-ups of a connected matching.

\begin{proposition}
\label{prop:cycles-in-blowups}
Let $0<\varepsilon\leq 1/18$, $p\in(0,1)$,  $M>1$ and $k\in\mathbb N$, and let $1/m\ll \varepsilon,1/M,1/k$. Let $\ell\in\mathbb N$ satisfy $20Mk^2\log_2 \varepsilon m \leq \ell \leq (1-50\varepsilon)km$, and let $H$ be a graph on at most $Mk$ vertices that contains a connected matching with at least $2k$ vertices. If $G$ is a $(\varepsilon, p, m)$-blow-up of $H$, then $G$ contains a copy of $C_{2\ell}$.
Moreover, if the connected matching is odd, then $H$ contains a copy of $C_{2\ell+1}$.
\end{proposition}

The strategy in the proof of Proposition \ref{prop:cycles-in-blowups} is to break the cycle into small paths that will be embedded separately in the $(\varepsilon,p)$-regular pairs given by the edges of the matching.
We then build the cycle by connecting these paths using Fact~\ref{fact:good-walk}. Two technical difficulties arise at this point. Firstly, the connections must avoid those vertices used to construct the paths, and also those vertices used to connect the paths. We will solve this by using only a few vertices in each cluster to make the connections. Secondly, the diameter of the blow-up of $G$ may be of order $\log m$, which means that these connections have to account for this obstacle. We address the second issue by defining the following graphs.

\begin{definition} \label{definition:doublebroom}
	For $h, s\in\mathbb N$, an \emph{$(h,s)$-double-broom} is a tree constructed by taking two disjoint copies of a binary tree of height $h$, and attaching their roots to the ends of a path with $s$ vertices, whose internal vertices are disjoint from both binary trees.
	We say that the two sets of leaves of the binary trees are the \textit{end sets} of the double-broom,
	and say that a double-broom \emph{starts} and \emph{ends} at sets $A$ and $B$, respectively, if one of its end sets is contained in $A$ and the other one in $B$.
\end{definition}

\begin{fact}
\label{fact:double-brooms}
Let $F_i$, $i\in\{1,2\}$, be vertex-disjoint copies of an $(h, s_i)$-double-broom in a graph $G$, with end sets $(A_i,B_i)$, respectively. If $e(B_1,A_2)>0$, then $G[V(F_1)\cup V(F_2)]$ contains an $(h, s_1+s_2+2h)$-double-broom with ends $(A_1,B_2)$. 
\end{fact}

We will use Fact \ref{fact:double-brooms} to combine several double-brooms to make longer double-brooms. Note that the longest paths in a double-broom are those paths between any pair of vertices belonging to different end sets. Also, note that once we have found an $(h,s)$-double-broom with end sets $A$ and $B$ such that $2h+s=\ell$, we only need to guarantee that $e(A,B)>0$ in order to find a copy of $C_\ell$.
We are now set to prove Proposition~\ref{prop:cycles-in-blowups}.

\begin{proof}[Proof of Proposition \ref{prop:cycles-in-blowups}]
Let $H$ be a graph on at most $Mk$ vertices which contains a connected matching on $2k$ vertices, and let us assume, without loss of generality, that $H$ has only one connected component.
Let $G$ be an $(\varepsilon,p,m)$-blow-up of $H$, let $(x_i,y_i)_{i\in [k]}$ be a matching in $H$, and let $\mathcal{M}=(X_i,Y_i)_{i\in [k]}$ be its counterpart in $G$.
If $H$ is bipartite, we assume the labelling is done in such a way that the $x_i$'s are all in the same part and consequently the same holds for the $y_i$'s.
Our strategy is to embed almost-spanning double-brooms into each of the regular pairs coming from the matching, and then use the structure of $G$ to connect those brooms. We will focus here when $H$ is non-bipartite and make explicit the differences in the bipartite case if necessary.

For each $i\in[k-1]$, let $Q_i$ be a good even-walk in $H$ between $y_i$ and $x_{i+1}$, and let $Q_k$ be a good odd-walk from $y_k$ to $x_1$.
(In the even case, we could take $Q_k$ to have even length, 
which would make the proof slightly simpler.)
We start by finding double-brooms inside the regular pairs given by the matching.

\begin{claim}
\label{claim:brooms-in-the-matching}Let $h=\lceil\log_2  \varepsilon m\rceil$. There exists a collection of $(h, s_i)$-double-brooms $\{F_i\}_{i\in [k]}$ in $G$ such that the following properties hold.
    \begin{enumerate}
        \item \label{item:broomsareintherightplace} $F_i\subseteq G[X_i,Y_i]$ and $F_i$ starts at $X_i$ and ends at $Y_i$ for each $i\in[k]$,
        \item \label{item:broomsarenotsolarge} $|F_i| \leq (2-96\varepsilon)m$ for each $i\in [k]$, and
        \item \label{item:broomshavetherightsizes} $\sum_{i=1}^{k}(2h+s_i)= 2\ell -  (h+1)\cdot\left( |Q_{k}|-3+\sum_{i=1}^{k-1}(|Q_i|-2)\right)$.
    \end{enumerate}
\end{claim}
\begin{proofclaim}
Partition $2\ell+1$ into $k$ numbers $\{t_i\}_{i\in[k]}$, as balanced as possible, so that $t_i$ is even for each $i\in[k-1]$ and $t_{k}$ is odd. In this way, these numbers differ by at most $2$ from each other, and they satisfy $2 \ell/k + 3 \geq t_i \geq 2 \ell/k - 2$.
Now define $s_i=t_i-2h-(h+1)(|Q_i|-2)$, for $i\in[k-1]$, and $s_{k}=t_{k}-2h-(h+1)(|Q_{k}|-3)-1$.
Note that each $s_i$ is even (since $|Q_i|$ is even for each $i\in [k-1]$ and $|Q_k|$ is odd).
Also, since $Q_i$ and $Q_k$ are walks on at most $3 |H| \leq 3Mk$ vertices, we have 
\[ s_i \geq t_i -2h - (h+1)3Mk \geq \frac{2 \ell}{k} - (h+1)(3Mk+2) \geq 2, \]
where in the last inequality we used $\ell \geq 20 M k^2 h$.
Note that the choices of $s_i$ satisfy precisely property \ref{item:broomshavetherightsizes} of the claim.
Moreover, the number of vertices of an $(h,s_i)$-double-broom is at most
\[2^{h+2}+s_i\leq 4\varepsilon m + \frac{2\ell}{k} \leq (2-96\varepsilon)m,\]
where we used $\ell \leq (1 - 50 \varepsilon)km$ in the last step. Thus we have the required property \ref{item:broomsarenotsolarge} of the claim.

Then we just have to find an $(h,s_i)$-double-broom $F_i\subseteq G[X_i,Y_i]$. Notice that double-brooms are trees with maximum degree $D=3$, as $s_i \geq 2$. Recall that $(X_i,Y_i)$ is an $(\varepsilon,p)$-regular pair in $G$ with $d_p(X_i,Y_i)\ge \varepsilon$ and $\varepsilon\leq 1/(4D+6)= 1/18$. Moreover, since each $s_i$ is even, the double-brooms are balanced bipartite graphs, with each part having at most $|P_i|/2 \leq (2-96\varepsilon)m/2
\le m- (2D+1)\varepsilon m$  vertices. Therefore, by applying Lemma \ref{lemma:regular-tree-embedding} we get the desired double-brooms.
\end{proofclaim}


\begin{claim}
\label{claim:connecting-brooms}
 Let $(W_i)_{i\in [2t]}$ be a sequence of clusters of $G$ such that, for $i\in[2t-1]$, the pair $(W_i,W_{i+1})$ corresponds to an edge in $H$. Then, for every choice of disjoint sets $W'_i\subset W_i$ with $|W'_i|\geq 10\varepsilon m$, $G[\bigcup_{i\in[2t]}W'_i]$ contains an $(h, 2t(h+1) - 2h)$-double-broom that starts at $W_1$ and ends at $W_{2t}$,
 and uses at most $4 \varepsilon m$ vertices from each $W'_i$.
\end{claim}

\begin{proofclaim}
We first check that the pairs $(W'_{2i-1},W'_{2i})$, $i\in[t]$, satisfy the hypothesis of Lemma~\ref{lemma:regular-tree-embedding} to find an $(h,2)$-double-broom.
Indeed, each double-broom has maximum degree $D = 3$, $\varepsilon<1/18$, $|W'_{2i-1}|,|W'_{2i}|\geq 10\varepsilon m = (2D+4)\varepsilon m$ and each colour class of an $(h,2)$-double-broom has at most $2\varepsilon m \leq |W'_i|-7\varepsilon m$ vertices.
Thus we may find an $(h,2)$-double-broom  $T_i \subseteq G[W'_{2i-1}, W'_{2i}]$ for each $i \in [t]$.
Note that each end-set of these double-brooms has at least $\varepsilon m$ vertices.
So, if we consider $T_i$ ending in $B_i\subset W'_{2i}$ and $T_{i+1}$ starting at $A_{i+1}\subset W'_{2i+1}$, we have
\begin{equation}
\label{equation:joining-brooms}
e(B_i,A_{i+1})>(d_p(W_{2i},W_{2i+1})-\varepsilon)|W_{2i}||W_{2i+1}|>0.
\end{equation}
Therefore, by Fact \ref{fact:double-brooms} we can find an $(h, 2t(h+1) - 2h)$-double-broom.
\end{proofclaim}

Now we are ready for the final part of the proof.
For $i\in[k]$, let $A_i\subset X_i$ and $B_i\subset Y_i$ be the end sets of $F_i$. We will say that a path $P$ \textit{connects} $F_i$ to $F_{i+1}$ (indices modulo $k$) if one end of this path is adjacent to $B_i$ and the other end has a neighbour $A_{i+1}$. We aim to find paths $P_i$ connecting $F_i$ to $F_{i+1}$, such that the paths $(P_i)_{i\in [k]}$ are pairwise disjoint and are also disjoint from $\bigcup_{i\in [k]}F_i$.
In addition, for each $i\in [k-1]$ we have $|P_i|=(h+1)(|Q_i|-2)$ and $|P_k|=(h+1)(|Q_k|-3)+1$.

Suppose we have already found $P_1,P_2,\dots, P_{j-1}$ for some $1\leq j\leq k$ (where the case $j=1$ is vacuous), and so that our task is to build $P_j$.
We denote by $(B_j, W_1,W_2,\ldots, W_{2t}, A_{j+1})$ the sequence of clusters coming from the walk $Q_j$, if $j<k$; and we let $(B_{k},W_1, W_2, \ldots , W_{2t+1}, A_1)$ be the corresponding sequence for $Q_{k}$. Let $W'_i\subset W_i$ be the vertices of $W_i$ that do not intersect $\bigcup_{i\in[k]}F_i$ or $\bigcup_{i\in[j-1]}P_i$.
Using Claim~\ref{claim:brooms-in-the-matching}\ref{item:broomsarenotsolarge}, the sizes of the paths and that $m$ is large, we conclude that
\begin{equation*}
|W'_i| \geq |W_i|- (2-96 \varepsilon)\frac{m}{2} - \sum_{i=1}^{j-1}|P_i| \geq 48\varepsilon m - k (h+1) 3Mk\geq 30\varepsilon m.
\end{equation*}
Moreover, as each vertex appears at most three times in the walk $Q_j$, we can take pairwise disjoint subsets $W''_i\subset W'_i$ with $10\varepsilon m$ vertices such that 
the pair $(W''_i,W''_{i+1})$ comes from an edge from $H$ for each $i\in [2t]$, if $j<k$, or $i\in[2t+1]$, if $j=k$. Then, by Claim~\ref{claim:connecting-brooms}, we get that $G[\bigcup_{i\in [2t]}W''_i]$ contains an $(h,2t(h+1)-2h)$-double-broom that starts at $W_1$ and ends in $W_{2t}$.

Now the analysis differs depending on the value of $j$.
Suppose first that $j < k$.
Since $(Y_j, W_1)$ and $(W_{2t},X_{j+1})$ both come from an edge in $H$ and the ends of the double-brooms have at least $\varepsilon m$ vertices, we can use the same argument as in \eqref{equation:joining-brooms} to find the desired path $P_j$.
Since $Q_j$ is a good even-walk between $y_j$ and $x_{j+1}$ in $H$, we have that $2t+2 = |Q_j|$,
and therefore $P_j$ has precisely $2t(h+1) = (h+1)(|Q_j|-2)$ vertices. 

We are only left with the case when $j=k$. Again, we have that $G[\bigcup_{i\in [2t]}W''_i]$ contains an $(h,2t(h+1)-2h)$-double-broom starting at $A'_{k}\subset W_1$ and ending at $B'_{k}\subset W_{2t}$. 
As before, we can find an edge between $B_{k}\subset Y_{k}$ (the end set of $F_{k}$) and $A'_{k}$, as both have at least $\varepsilon m$ vertices and since they come from an edge in $H$. To connect to $F_1$, however, we cannot use a single edge and thus need to make a length 2 connection via the set $W''_{2t+1}$. To do so, we need to find a vertex $u\in W''_{2t+1}$ that has a neighbour in both $B'_{k}$ and $A_1$. By Lemma~\ref{lemma:regularpairs}, the number of vertices in $W_{2t+1}$ with no neighbours in $B'_{k}$ is at most $\varepsilon m$, and the number of vertices in $W_{2t+1}$ with no neighbours in $A_1$ is also at most $\varepsilon m$.
Therefore, since $|W''_{2t+1}|> 2\varepsilon m$, we can find such a vertex $u$.

Since $Q_k$ is a good-odd walk, we have $|Q_k| = 2t+3$ and thus $P_k$ has precisely $2t(h+1) + 1 = (h+1)(|Q_k| - 3) + 1$ vertices. Then, using the double-brooms $F_i$'s and the paths $P_i$'s connecting $F_i$ with $F_{i+1}$, we find a cycle of length
\[\sum_{i=1}^{k}(2h+s_i)+\sum_{i=1}^{k-1}(h+1)(|Q_i|-2) +(h+1)(|Q_{k}|-3)+1=2\ell+1,\]
which finishes the proof.
\end{proof}

\subsection{Proof of the Stability lemma}The last two lemmas that we need are well-known Tur\'an-type results for cycles. 
\begin{lemma}[Erd\H os--Gallai~\cite{erdHos1959maximal}]\label{lemma:ErdosGallai} For all $k\ge 3$ and $n\ge 1$, every $n$-vertex graph with no cycle of length at least $k$ has at most $\frac{1}{2}(k-1)(n-1)$ edges.
\end{lemma}
\begin{lemma}[Jackson~\cite{JACKSON1985118}]\label{lemma:jackson}Let $n,m,t\in\mathbb N$ satisfy $n\ge m\ge t\ge 2$ and $m\le 2t-2$. Suppose $G$ is a bipartite graph with parts of size $n$ and $m$ such that $e(G)>(n-1)(t-1)+m$. Then, $G$ contains a cycle of length at least $2t$.
    
\end{lemma}

We are now ready to prove Theorem~\ref{theorem:stability:1}. 
The proof is divided into three steps: starting from $G = G(N,p)$, we apply regularity; then we deduce that the reduced graph should follow an extremal colouring; and finally we transfer that information to $G$.

\begin{proof}[Proof of Theorem~\ref{theorem:stability:1}]
During the proof we will choose constants in the following hierarchy
\[1/n_0\ll1/K\ll \eta  \ll \varepsilon \ll  \delta \ll \delta'\ll \alpha' \ll \alpha \ll 1.\]

\noindent \emph{Step 1: Applying regularity.}
Let $G=G(N,p)$.
By Lemma \ref{lemma:upper-uniform}, we know that for $p\geq K/N$, w.h.p. $G$ is $(\eta,p)$-uniform, thus, in particular, we have  $e(G) = (1\pm \eta) p \binom{N}{2}$. Let us consider a red-blue colouring of $E(G)$, and let $G_R$ and $G_B$ be the graphs formed by the red and blue edges, respectively.
By assumption, neither $G_R$ nor $G_B$ contain a copy of $C_n$.
Since both $G_R$ and $G_B$ are $(\eta,p)$-upper uniform, we may use the Colourful sparse regularity lemma (Theorem~\ref{theorem:regularity}) in order to find an $(\varepsilon,p)$-regular partition $V(G)=V_0\cup V_1\cup\dots \cup V_k$, with $k\geq 1/\varepsilon$, so that all but at most $\varepsilon k^2$ pairs $(V_i,V_j)$ are $(\varepsilon,p)$-regular in both $G_R$ and $G_B$.

Let $\Gamma_R$ and $\Gamma_B$ be the $(\varepsilon, p,\varepsilon)$-reduced graph of $G_R$ and $G_B$, respectively, and set $\Gamma=\Gamma_R\cup \Gamma_B$. 
For all but at most $\varepsilon k^2$ choices, the pair $(V_i, V_j)$ is $(\varepsilon, p)$-regular both for $G_R$ and $G_B$; we will show that any such pair belongs to $\Gamma$.
Indeed, since $G$ is $(\eta, p)$-uniform, we know that for every such pair, $G[V_i, V_j]$ contains at least $(1 - \eta)p|V_i||V_j|$ edges.
Therefore, in the most popular colour $c$ among those used in $G[V_i, V_j]$,
we certainly have $G_c[V_i, V_j] \geq (1 - \eta)p|V_i||V_j|/2 \geq \varepsilon p |V_i||V_j|$, and therefore the pair $(V_i, V_j)$ belongs to $\Gamma_c \subseteq \Gamma$.
We deduce then that
\[e(\Gamma)\ge (1-5\varepsilon)\frac{k^2}{2}.\]
Now we colour the edges of $\Gamma$ by declaring an edge red or blue if it belongs to $\Gamma_R$ or $\Gamma_B$, respectively, where ties are broken arbitrarily.

\medskip \noindent \emph{Step 2: Extremal colouring in the reduced graph.} In this part, we will show that the bound $N\geq (1-\delta)R(C_n)$ implies that this colouring of $\Gamma$ is either $\alpha'$-odd- or $\alpha'$-even-extremal, depending on the parity of $n$.

Assume that $n$ is odd first, and note that $\delta'\ll \alpha'$ is enough to apply Proposition~\ref{proposition:stability:odd} with $\delta', \alpha'$ in place of $\delta$ and $\alpha$.
Let $k_1$ be the largest odd number such that $k\geq (2-\delta')k_1$.
Proposition \ref{proposition:stability:odd} states that the colouring of $\Gamma$ either is $\alpha'$-odd-extremal or it contains a monochromatic copy of $C_{k_1}$. 
We would like to rule out this last possibility.
If it contains a monochromatic $C_{k_1}$, say in red, then $G_R$ contains an $(\varepsilon, p, m)$-blow-up of an odd connected matching on at least $k_1-1$ vertices, for some $m\geq \lfloor (1-\varepsilon)N/k\rfloor$.
Therefore, by applying Proposition \ref{prop:cycles-in-blowups}, with $M = 3$, we get that $C_n\subset G_R$, as long as
\begin{equation}\label{eq:odd component}240k^2\log_2 \varepsilon m \leq n-1 \leq (1-50\varepsilon)(k_1-1)m.\end{equation}
The lower bound in~\eqref{eq:odd component} holds as $N= O(n)$ and $n\gg \log n$, for sufficiently large $n$. Moreover, we will assume this for all future applications of Proposition \ref{prop:cycles-in-blowups}, so we will not check again if this inequality holds.  For the upper bound in~\eqref{eq:odd component}, first notice that $k_1-1\geq k/(2-\delta')-3\geq (1+\delta'/2)k/2$, since $k\geq 1/\varepsilon$. Also, $m\geq (1-2\varepsilon)N/k \geq (1-\delta-2\varepsilon)2n/k$, and thus 
\begin{align*}
(1-50\varepsilon)(k_1-1) m \geq (1-50\varepsilon) \left(1+\frac{\delta'}2\right)\frac{k}{2} \frac{(1-\delta-2\varepsilon) 2n}{k} \geq& \left(1-\delta-52\varepsilon+ \frac{\delta'}{2}\right)n \geq \ n,
\end{align*}
as long as $ \delta, \varepsilon \ll \delta'$.
In the above inequality, we used that $N\geq (1-\delta)2n$ for the odd case and that we could find odd cycles of length roughly $k/2$.
The conclusion is that $\Gamma$ cannot contain a monochromatic copy of $C_{k_1}$, and therefore the colouring of $\Gamma$ is $\alpha'$-odd-extremal, as required.

The argument changes straightforwardly when considering the case of even $n$, so we omit it.
In any case, the conclusion we get is that the red-blue colouring of $E(\Gamma)$ is $\alpha'$-odd-extremal, or $\alpha'$-even-extremal, according to the parity of $n$. 

\medskip \noindent \emph{Step 3: Extremal colouring in the random graph.}
Now that we have established that $\Gamma$ follows an extremal colouring, we will show that $G$ inherits this colouring.
We consider a partition $V(\Gamma)=X\cup Y$, with $|X|\geq |Y|$, given by the definition of extremal colourings (red colour inside parts and blue across parts).
We also consider the partition $V(G)=X'\cup Y'$ given by the union of the corresponding clusters contained in $X$ or in $Y$, respectively. 

\begin{claim}\label{claim:biparte:alpha}
$e(G_R[X',Y']) < \alpha e(G[X',Y'])/2$.
\end{claim}
\begin{proofclaim}
   For the sake of contradiction, assume that $e(G_R[X',Y']) \geq \alpha e(G[X',Y'])/2$.    We will show that $\Gamma_R$ contains an odd connected matching on at least $|X|+\alpha k/400$ vertices, in which case we can find a copy of $C_n$ in red by using Proposition~\ref{prop:cycles-in-blowups}, contradicting that $G$ contains no monochromatic $C_n$.
    
    Firstly, we will prove that $\Gamma_R(X,Y)$ is dense. As $G$ is $(\eta,p)$-uniform, we have that 
    \[e_G(X',Y')\geq (1-\eta)p|X'||Y'| \geq \frac{(1-\alpha')^2pN^2}{18}\geq \frac{pN^2}{36},\]
    provided $\eta,\alpha' \ll 1$. By discounting the edges from $G$ which are not present in $\Gamma$, we have 
    \[\frac 1pe_{G_R}(X',Y')-10\varepsilon N^2\le \sum_{(A,B)\in \Gamma_R[X,Y]}d_p(A,B)|A||B|\le e(\Gamma_R[X,Y])\left(\frac{N}{k}\right)^2,\]
    and thus
    \begin{equation}\label{eq:red-edges-crossing}e(\Gamma_R[X,Y])\ge \left(\frac{k}{N}\right)^2\cdot \frac{\alpha}{2}\cdot\left(\frac{N^2}{36}-10\varepsilon N^2\right)\ge \frac{\alpha k^2}{100}.
    \end{equation}
   We claim that~\eqref{eq:red-edges-crossing} implies $\Gamma_R[X,Y]$ contains a path $P$ with $\alpha k/100$ vertices. Indeed,~\eqref{eq:red-edges-crossing} implies that the average degree in $\Gamma_R[X,Y]$ is at least $\alpha k/100$ and thus there exists a subgraph $\Gamma'\subset \Gamma_R[X,Y]$ with $\delta(\Gamma')\ge \alpha k/200$. As $\Gamma'$ is bipartite and $\delta(\Gamma')\ge \alpha k/200$, we can greedily find $P\subset \Gamma'$. Set $\tilde{X} = X\setminus V(P)$, and let us show now that $\Gamma_R[\tilde{X}]$ is almost complete. As the colouring of $\Gamma$ is $\alpha'$-extremal, we have
  \begin{equation}
        \label{eq:edges-complement-gamma-red}
        e(\Gamma_R^c[X])\leq \alpha' \binom{|X|}{2} +5\varepsilon k^2\leq 2\alpha' k^2,
    \end{equation}
and thus 
\[e(\Gamma_R[\tilde{X}])\geq \binom{|X|}{2}-2\alpha'k^2\ge \binom{|\tilde{X}|}{2}-2\alpha'k^2.\]
Therefore, by Lemma~\ref{lemma:ErdosGallai}, we can find a cycle $C\subset \Gamma_R[\tilde{X}]$ such that 
\begin{equation}\label{eq:size-red-cycle}
|V(C)|\geq \frac{2e(\Gamma_R[\tilde{X}])}{|\tilde{X}|-1}\ge\frac{2}{|\tilde{X}|-1}\left(\binom{|\tilde{X}|}{2} - 2\alpha'k^2\right)\ge |\tilde{X}| - 20\alpha'k, \end{equation}
where we used that $|\tilde{X}|-1\geq k/5$ if $\alpha'$ is small enough. Let $P'$ be a maximal subpath of $P$ with both ends in $X$ such that both endpoints of $P'$ have more than $|V(C)|/2$ red neighbours in $V(C)$. If $P'$ is obtained by removing one end vertex at a time, one can conclude that $|V(P)\setminus V(P')|\leq 2|X\cap (V(P)\setminus V(P'))|+2$. Using the definition of $P'$ and~\eqref{eq:edges-complement-gamma-red}, we get
\[2\alpha' k^2 \geq |X\cap (V(P)\setminus V(P'))|\cdot \frac{|V(C)|}{2} \geq |X\cap( V(P)\setminus V(P'))| \cdot \frac{k}{10},\]
using that $|\tilde{X}|\geq k/4$ provided $\alpha'\ll\alpha$ are small enough. Therefore, we have $|V(P)\setminus V(P')|\leq 50\alpha' k$. Let $u$ and $v$ be the endpoints of $P'$.  Since both $u$ and $v$ have more than $|V(C)|/2$ red neighbours in $C$, there exists an edge $u'v'\in E(C)$ such that $uu', vv' \in \Gamma_R$, which allows us to find a cycle on at least
\begin{align*}
|C|+|P'| &\geq |\tilde{X}|-20\alpha'k+|P|-50\alpha' k\ge |X|+\frac{\alpha k}{200}-70\alpha' k\ge |X|+\frac{\alpha k}{400}
\end{align*}
vertices. Note that this cycle is a connected matching on at least $|V(C)|-1$ vertices. Moreover, as $\Gamma_R[V(C)]$ spans at least $\binom{|V(C)|}{2}-2\alpha' k^2$ edges, it is easy to see that it must be non-bipartite. Finally, we may find a monochromatic copy of $C_n$ by applying Proposition~\ref{prop:cycles-in-blowups} if we prove that $n \leq (1-50\varepsilon)(|X|+\alpha k/400)m$. Indeed, using that $|X|N\geq (1-\alpha')(1-\delta)kn$ (for both parities) and that $m\geq (1-2\varepsilon)N/k$, we get that 
\begin{equation}
    \label{eq:finally}
(1-50\varepsilon)|X|m \geq (1-50\varepsilon)(1-\alpha')(1-2\varepsilon)(1-\delta)n\geq (1-2\alpha')n, 
\end{equation}
provided that $\varepsilon\ll \delta\ll\alpha'$. Using that $(1-50\varepsilon)m \geq n/2k$ for small enough $\varepsilon$, we conclude that
\[ (1-50\varepsilon)\left(|X| + \frac{\alpha k}{400}\right)m \geq (1-2\alpha')n + \frac{\alpha}{800}n \geq n,\]
as long as $\varepsilon\ll\alpha'\ll\alpha$. This finishes the proof of this claim.
\end{proofclaim}

Now the proof splits according to the parity of $n$, due to the inherent difference between these two problems. In the odd case, the cycle cannot be embedded in the blue bipartite graph, while in the even case, it is a matter of space.
\begin{claim}
\label{claim:blue-edges-inside-odd}
If $n$ is odd, then $e(G_B[X']) < \alpha e(G[X'])/2$ and $e(G_B[Y']) < \alpha e(G[Y'])/2$.
\end{claim}
\begin{proofclaim} Suppose by contradiction that $e(G_B[X'])\geq \alpha e(G[X'])/2$ (the case obtained by replacing $X'$ with $Y'$ is analogous).
    We will use this fact to show that $\Gamma_B$ contains an odd connected matching with at least $2|Y|-20\alpha' k$ vertices, which we may use to find a blue copy of $C_n$ in $G$.
    
    We first prove that $\Gamma_B[X]$ is dense. Again, since $G$ is $(\eta,p)$-uniform, we have that
\[e(G[X'])\geq \frac{p}{2}\binom{|X'|}{2} \geq \frac{p|X'|^2}{8}\geq \frac{pN^2}{64}.\]
As in \eqref{eq:red-edges-crossing}, this lower bound implies that $e(\Gamma_B[X])\geq \alpha k^2/100$, and, by the same argument as in \eqref{eq:edges-complement-gamma-red}, we have that $e(\Gamma_B^c[X,Y])\leq 2\alpha' k^2$. Then, by Lemma~\ref{lemma:jackson}, we can find a cycle $C\subset \Gamma_B[X,Y]$ with 
\begin{align*}
    |V(C)|\geq \frac{2e(\Gamma_B[X,Y])}{|X|}-O(1)\ge \frac{2}{|X|}\cdot \left(|X||Y| - 2\alpha'k^2 \right) -O(1) \geq 2|Y| - 20\alpha'k,
\end{align*}
using that $|X|\geq k/4$ and that $k$ is large. We claim that $C$ belongs to a non-bipartite component of $\Gamma_B$. Note that because of the lower bound on $|Y|$, we have that $|X|-|Y| \leq 2\alpha' k$ and then
\[|X\setminus V(C)|= |X|-|V(C)|/2 \leq |X|-|Y|+10\alpha' k \leq 12\alpha' k.\]
Therefore, the number of edges of $\Gamma_B[X]$ touching $X\setminus V(C)$ is at most $12\alpha' k^2$. Since $\alpha'\ll \alpha$, we have $e(\Gamma_B[X])\geq \alpha k^2/100 > 12\alpha' k^2$ and therefore $V(C)\cap X$ contains a blue edge. This implies that $V(C)$ induces a large odd-matching in $\Gamma_B$.

Now we may apply Proposition~\ref{prop:cycles-in-blowups} to find a copy of $C_n$ in $G_B$, which leads to a contradiction. It is thus enough to check that $n\leq (1-50\varepsilon)(2|Y|-20\alpha' k)m$. Indeed, as $2|Y|\geq (1-\alpha')k$ and $m\geq (1-2\varepsilon)(1-\delta)2n/k$, we have
\[(1-50\varepsilon)(2|Y|-20\alpha' k)m \geq (1-30\alpha')2n\geq n,\]
since $\varepsilon\ll \delta\ll\alpha' \ll 1 $.
Thus, we have shown that $e(G_B[X']) < \alpha e(G[X'])/2$.
\end{proofclaim}
Now we move to the even case.
\begin{claim}
If $n$ is even, then $e(G_B[X']) < \alpha e(G[X'])/2$.
\end{claim}
\begin{proofclaim}
    By contradiction, let us assume that $e(G_B[X']) \geq \alpha e(G[X'])/2$. Again, our aim is to find a blue copy of $C_n$ in $G$.
    First, we will show that $\Gamma_B$ contains a connected matching on at least $2|Y|+ \alpha k/400$ vertices.
    Similarly as in the previous cases, one can show that $e(\Gamma_B[X]) \geq \alpha k^2/100$ and that $e(\Gamma_B^c[X,Y]) \leq 2\alpha' k^2$. Therefore, we may find a path $P\subset \Gamma_B[X]$ with $\alpha k/100$ vertices. Let $\tilde{X}=X\setminus V(P)$. Using Lemma~\ref{lemma:jackson}, we can find a cycle $C\subset \Gamma_B[\tilde{X},Y]$ such that 
    \[|V(C)|\geq  \frac{2e(\Gamma_B[\tilde{X},Y])}{|\tilde{X}|}-O(1)\ge \frac{2}{|\tilde{X}|} \left(|\tilde{X}||Y| - 2\alpha' k^2\right) - O(1) \geq 2|Y| - 30\alpha'k,\]
    \noindent  using that $|\tilde{X}|\geq k/4-|P| \geq k/5$ and that $k\geq 1/\varepsilon$. Now we take a maximal path $P'\subset P$ such that both endpoints of $P'$ have at least $|V(C)|/4$ blue neighbours in $V(C)\cap Y$. By the upper bound on $e(\Gamma_B^c[X,Y])$, we have
    \[2\alpha' k^2 \geq |V(P)\setminus V(P')| \cdot \frac{|V(C)|}{4} \geq |V(P)\setminus V(P')|\cdot \frac{k}{15},\]
and thus $|V(P)\setminus V(P')|\leq 30\alpha' k$. By the same reasoning as in Claim~\ref{claim:blue-edges-inside-odd}, there exist vertices $u',v'\in V(C)\cap Y$ at distance $2$ in $C$ such that the ends $u,v$ of $P'$ satisfy $uu',vv' \in \Gamma_B$. This yields a cycle on $|V(C)|+|V(P')|-1$ vertices. Since this cycle might be odd, we find a connected matching with at least
    \[2|Y|-30\alpha' k + \frac{\alpha}{160}k - 30\alpha'k - 2 \geq 2|Y| + \frac{\alpha}{400}k\]
vertices. Now, by checking that $2|Y|N\geq (1-\alpha')(1-\delta)kn$, the proof follows by applying Proposition~\ref{prop:cycles-in-blowups} (doing the same calculation as in \eqref{eq:finally}).
\end{proofclaim}
To finish the proof, we set $X''=X'\cup V_0$ and show that the partition $X''\cup Y'$ verifies that the colouring is $\alpha$-even/odd-extremal. Let $n$ be odd, and recall that each cluster of $\Gamma$ has $m\geq (1-2\varepsilon)N/k$ vertices. By the definition of $X'$ and $Y'$, 
\[|X'|,|Y'|\ge (1-\alpha')\frac{k}{2}m \geq (1-\alpha')(1-2\varepsilon )\frac{N}{2}\geq (1-\alpha)\frac{N}{2},\]
since $\varepsilon \ll\alpha'\ll \alpha$. By Claim~\ref{claim:blue-edges-inside-odd}, we have that $e(G_B[Y'])\leq \alpha e(G[Y'])$. Now let $E_0$ be the set of edges touching $V_0$.  Since $G$ is  $(\eta,p)$-uniform, we have that $|E_0|\leq 2\varepsilon pN^2$, and therefore 
\[e(G_B[X''])\leq e(G_B[X']) + |E_0| \leq \frac{\alpha}{2}e(G[X'']) + 2\varepsilon pN^2 \leq \alpha e(G[X'']),\]
as $\varepsilon \ll \alpha \ll 1$ and since $e(G[X'']) \geq (1-\eta)p\binom{|X''|}{2} \geq pN^2/16$. A similar argument shows that $e(G_R[X'',Y'])\leq \alpha e(G[X'',Y'])$, which proves that the colouring is $\alpha$-odd-extremal. The case when $n$ is even follows essentially the same proof.\end{proof}

\section{Rotation-extension} \label{section:rotation}

\subsection{Boosters in expanders}

To find long cycles in expander subgraphs, we will use the well-known rotation-extension technique pioneered by Pósa~\cite{Posa1976}.

A pair $\{u,v\} \in V^{(2)}$ is a \emph{booster} in a graph $G$, if $G + uv$ is Hamiltonian or its longest path is longer than that of $G$.
Observe that if $G$ is connected and non-Hamiltonian and $P$ is a longest path in $G$ from $u$ to $v$, then $\{u,v\}$ is a booster.
Indeed, $P + \{u,v\}$ closes to a cycle $C$.
If $|V(C)| = |V(G)|$, then we have a Hamilton cycle; otherwise, as $G$ is connected,  $C$ is adjacent to a vertex outside of $V(C)$, so we can find a path $P'$ which contains all of $V(C)$ and has more vertices than $|V(C)| = |V(P)|$.

\begin{lemma} \label{lemma:tapaoenboosters}
     Let $d, k, M, m, n$ satisfy $d \geq 2$, $k+1 \geq m$, $m \leq M$ and $M \leq n/4$.
     Suppose $G$ is an $n$-vertex graph which is a $(k,d)$-expander and $(m,M)$-joined.
     If $G$ is not Hamiltonian, then $G$ has at least $\frac{1}{16} \binom{n}{2}$ boosters.
\end{lemma}

To prove \cref{lemma:tapaoenboosters}, we recall the basics of the rotation-extension technique.
Let $P$ be a path in a graph $G$, whose endpoints are $u$ and $v$, and we consider its vertices to be ordered so that $u$ is its first vertex.
Given $x \in V(P)$, we write $x^{-}$ for the vertex before $x$ on $P$, and $x^{+}$ for the vertex after $x$ on $P$, if they exist.
For $X \subseteq V(P)$, we write $X^{-} = \{ x^{-} : x \in X \}$ and $X^{+} = \{ x^{+} : x \in X \}$. Note that if $x \in V(P)$ is a neighbour of $v$, then $P - xx^{+} + vx$ is a path in $G$ from $u$ to $x^{+}$ whose vertex set is $V(P)$.
We say such a path is obtained from $P$ by an \emph{elementary rotation}.
A path obtained from $P$ by a sequence of elementary rotations, with $u$ fixed, is a path \emph{derived} from $P$.
The set of ending vertices of paths derived from $P$, including $v$, will be denoted by $S(P, u)$.
Note that $S(P, u) \subseteq V(P)$, and that for every $w \in S(P,u)$ there exists a path in $G$ of the same length as $P$ with endpoints $u$ and $w$.
The following result appears in~\cite[Lemma 2.6]{BBDK2006}.

\begin{lemma} \label{lemma:rotation}
    Let $P$ be a longest path in a graph $G$ with endpoints $u$ and $v$, and let $S = S(P,u)$.
    Then $N(S) \subseteq S^{+} \cup S^{-}$.
    In particular (since $v^+$ does not exist), $|N(S)| < 2 |S|$.
\end{lemma}

\begin{proof}[Proof of \cref{lemma:tapaoenboosters}]
    First, we observe that $G$ is connected.
    Indeed, suppose not, and let $\{A, B\}$ be a partition of $V(G)$ without crossing edges.
    Suppose that $|A| \leq |B|$ and in particular that $|B| \geq n/2 \geq M$.
    Since $N(A) \setminus A = \emptyset$ and since $G$ is $(k,d)$-expander, then $|A| \geq k+1 \geq m$.
    Hence, any pair of sets $A' \subseteq A$ and $B' \subseteq B$ of sizes $m, M$ respectively have $e(A',B')=0$, which is a contradiction to the fact that $G$ is $(m, M)$-joined. This shows that $G$ is connected.

    Next, we show that for each vertex $u$ which is a start of a longest path $P$ in $G$, the bound $|S(P,u)| \geq (n-M+1)/3$ holds.
    This is enough to conclude.
    Indeed, we have that for each $v \in S(P,u)$, the pair $\{u,v\}$ is a booster (since $G$ is connected). 
    Thus we would have at least $(n-M+1)/3 \geq n/4$ booster pairs containing $u$, but also at least $n/4$ other endpoints of a longest path in $G$.
    These two pieces of information together imply that the number of boosters is at least $\frac{1}{16}\binom{n}{2}$, as desired.
    
    Let $P$ be a longest path in $G$ with endpoints $u$ and $v$, and let $S = S(P,u)$.
    As discussed, it is enough to show that $|S| \geq (n-M+1)/3$.
    By \cref{lemma:rotation}, we have $|N(S)| < 2|S|$.
    Since $G$ is a $(k,2)$-expander, this implies $|S| \geq k+1 \geq m$.
    Then Lemma~\ref{lemma:improvedexpansion} implies that $|N(S)| \geq n - M -|S|+ 1$.
    Therefore, we deduce $n - M + 1 \leq 3 |S|$ and thus $|S| \geq (n-M+1)/3$, as required.
\end{proof}

\subsection{Boosters in bipartite expanders}

We need to tailor the rotation-extension technique to work in bipartite graphs.
This has been done for balanced bipartite graphs, e.g. by Frieze~\cite{Frieze1985} and Bollobás and Kohayakawa~\cite{BollobasKohayakawa1991}, our treatment here is slightly different to allow for unbalanced graphs.

Given a bipartite graph $G$, with parts $V_1$ and $V_2$ such that $|V_1| \geq |V_2|$, we say that a pair $\{u,v\} \notin E(G)$ with $u \in V_1$, $v \in V_2$ is a \emph{bipartite booster for $G$}, if $G+uv$ contains a cycle of length $2|V_2|$ or its longest path is longer than that of $G$.
Given a set of vertices $S$, an \emph{$S$-path} is a path whose endpoints are in $S$.

\begin{lemma} \label{lemma:prebipartiteboosters}
    Let $d, k, m, n$ satisfy $d \geq 2$, $k+1 \geq m$ and $m \leq n$.
    Suppose $G$ is a bipartite graph, with parts $V_1$ and $V_2$ such that $|V_1| \geq |V_2|+m = n+m$, which is a $(k,d)$-bipartite-expander and $m$-bipartite-joined.
    Then,
    \begin{enumerate}
        \item \label{item:prebipartiteboosters-Gconnected} $G$ is connected,
        \item \label{item:prebipartiteboosters-V2path} if $P$ is a longest path in $G$, then $P$ is a $V_1$-path,
        \item \label{item:prebipartiteboosters-booster} if $G$ does not contain a $2|V_2|$-length cycle, then for any longest $V_1$-path $P = v_0 v_1 \dotsb v_\ell$, the pair $\{ v_1, v_{\ell} \}$ is a bipartite booster.
    \end{enumerate}
\end{lemma}

\begin{proof}
    To see~\ref{item:prebipartiteboosters-Gconnected}, suppose there exists a partition $\{A,B\}$ of $V(G)$ without crossing edges.
    Suppose $|A \cap V_1| \leq |B \cap V_1|$.
    In particular, $|B \cap V_1| \geq m$.
    Since $G$ is $m$-bipartite-joined, we see $|N(B \cap V_1)| > |V_2| - m$ and thus $|A \cap V_2| < m$.
    Since $G$ is a $(k,d)$-bipartite-expander, we have $|N(A \cap V_2)| \geq d|A \cap V_2| \geq |A \cap V_2|$.
    Select $A' \subseteq A \cap V_1$ of size precisely $|A \cap V_2| < m$. Then we have $N(A') \subseteq A \cap V_2$, but $|A \cap V_2| \geq |N(A')| \geq 2 |A'| = 2 |A \cap V_2|$, a contradiction.
    
    To see~\ref{item:prebipartiteboosters-V2path}, let $P$ be a longest path in $G$ and suppose it is not a $V_1$-path.
    Let $u, v$ be its endpoints, suppose $v \in V_2$, and let $S = S(P,u)$.
    Note that $S \subseteq V_2$, since $v \in V_2$.
    By \cref{lemma:rotation}, $|N(S)| < 2|S|$.
    Since $G$ is $(k,d)$-bipartite-expander and $d \geq 2$, we deduce $|S| \geq k+1 \geq m$.
    Note that $|V_1 \setminus V(P)| \geq |V_1| - |V_2| \geq m$.
    Since $G$ is $m$-bipartite-joined, there exists an edge between $S$ and $V_1 \setminus V(P)$.
    Hence, there is a longest path $P'$ with an endpoint having a neighbour outside of $V(P')$, which contradicts the maximality of~$P$.
    
    Finally, to see~\ref{item:prebipartiteboosters-booster}, let $P = v_0 v_1 \dotsb v_{\ell}$ be a longest $V_1$-path in $G$.
    Note that adding $\{v_1, v_{\ell}\}$ to $G$ gives a cycle $C$ of length $\ell$ and $V(C) = V(P) \setminus \{ v_0 \}$.
    If $\ell = 2|V_2|$, then $\{v_1, v_{\ell}\}$ is a bipartite booster, and so we may suppose $\ell < 2 |V_2|$.
    Note that $|V(P) \cap V_2| < |V_2|$, and so there exists a vertex $u \in V_2 \setminus V(P)$.
    Since $G$ is connected (by~\ref{item:prebipartiteboosters-Gconnected}), there exists a shortest path $Q$ starting in $u$ and ending in some vertex of $V(P)$.
    This path cannot end in $v_0$, as that would yield a longer path than $P$ in $G$, so it must end in $V(C)$.
    Merging $Q$ with $C$ yields a path of length at least $\ell$ in $G+v_1v_\ell$ which is not a $V_1$-path, which by~\ref{item:prebipartiteboosters-V2path}, means that $G+v_1v_\ell$ has a $V_1$-path of length at least $\ell+1$.
    We conclude that $G+v_1v_\ell$ has a path longer than that of $G$, and thus $\{v_1, v_\ell\}$ is a bipartite booster.
\end{proof}

\begin{lemma}
\label{lemma:bipartite-boosters}
    Let $d, k, m, n$ satisfy $d \geq 2$, $k+1 \geq m$, and $n \geq 5m$.
    Suppose $G$ is a bipartite graph with parts $V_1, V_2$ such that $|V_1| \geq |V_2|+m = n+m$, which is a $(k,d)$-bipartite-expander and $m$-bipartite-joined.
    If $G$ does not contain a cycle of length $2|V_2|$, then it contains at least $n^2/8$ bipartite boosters.
\end{lemma}

\begin{proof}
    Let $P$ be a longest path in $G$, starting at the vertex $u$.
    By \cref{lemma:prebipartiteboosters}\ref{item:prebipartiteboosters-V2path}, $P$ is a $V_1$-path.
    Recall that $S = S(P, u) \subseteq V_1$ is the set of ending vertices of paths derived from $P$.
    Let $T \subseteq V_2$ be the set of penultimate vertices of all paths derived from $P$.
    That is, if $u \dotsb wv$ is any path derived from $P$, then $w \in T$ and $v \in S$.
    Note that $T, S \subseteq V(P)$.
    We note that it is enough to show that $|S| \geq n/2$ and $|T| \geq n/4$.
    Indeed, by \cref{lemma:prebipartiteboosters}\ref{item:prebipartiteboosters-booster}, for each $w \in T$ we have that $\{u,w\}$ is a bipartite booster, so $|T| \geq n/4$ implies that $u$ is contained in at least $n/4$ bipartite boosters.
    But $|S| \geq n/2$ means that at least $n/2$ vertices are also endpoints of a longest path in $G$, so this means that those vertices are also contained in at least $n/4$ bipartite boosters each.
    Thus the number of bipartite boosters is at least $|S|n/4 \geq n^2/8$, as desired.
    
    \cref{lemma:rotation} implies that $|N(S)| < 2|S|$. As $G$ is $(k,2)$-bipartite-expander, then $|S| \geq k+1 \geq m$. Also, since $G$ is $m$-joined and $|V_1| \geq n+m$, then $|N(S)| \geq |V_1| - m + 1 > n$.
    Therefore $|S| > |N(S)|/2 > n/2$, as required.
    
    So it only remains to show that $|T| \geq n/4$. Given $x \in V(P)$, recall that $x^-, x^+$ are the predecessor and successor of $x$ with respect to the vertex ordering in
    $P$.
    We will show that
 if $x \in S$, then $x^{-} \in T$ or $x^+ \in T$.
    This easily implies that $|T| \geq |S|/2 \geq n/4$, as desired.
    
    To show 
    this,
    suppose, for a contradiction, that $x \in S$ but $x^-, x^+ \notin T$.
    Since $x \in S$, there is a sequence $P_0, P_1, \dotsc, P_t$ of paths derived from $P$, where $P_0 = P$, such that $P_t$ has endpoints $u$ and $x$, and for $1 \leq i \leq t$, the path $P_{i}$ is obtained from $P_{i-1}$ by an elementary rotation.
    Let $0\le j\le t$ be the maximum integer such that $x$ is adjacent to both $x^+$ and $x^-$ in $P_j$.
    Our assumptions imply that $0 \leq j < t$.
    Suppose $P_j = u \dotsb x_1 x x_2 \dotsb w$, with $\{x_1, x_2\} = \{x^-, x^+\}$.
    By the choice of $j$, $x$ is adjacent to both $x^+$ and $x^-$ in $P_j$, but not in $P_{j+1}$.
    Since $P_j$ is a longest path, it is a $V_1$-path, so the only possibility is that the edge $wx_1$ was used to rotate $P_j$ and obtain $P_{j+1}$.
    But this means that $x$ is the endpoint of $P_{j+1}$ and $x_2$ is the penultimate vertex of $P_{j+1}$.
    Thus $x_2 \in T$, a contradiction.
\end{proof}

\subsection{Boosters in subgraphs of random graphs}
In this section, we modify an argument by Lee and Sudakov~\cite{LS2012} which shows that, with high probability, every sufficiently sparse expander subgraph $H \subseteq G(n,p)$ has a booster in $G(n,p)$.
We need to modify their argument slightly to allow for bipartite boosters and to give a lower bound on the number of such objects appearing on $G(n,p)$, but otherwise the argument is the same.

Let $\mathcal{F}$ be a family of graphs on the same $n$-vertex set $V$ together with a family of subgraphs $\mathcal{B} = \{ B_F \subseteq \binom{V}{2} : F \in \mathcal{F} \}$, one for each $F \in \mathcal{F}$.
We say $(\mathcal{F}, \mathcal{B})$ is \emph{$(\delta , \alpha)$-boosterable} if
\begin{enumerate}
    \item for each $F \in \mathcal{F}$, $|E(F)| \leq \delta n^2$, and
    \item for each $F \in \mathcal{F}$,  $B_F \subseteq \binom{V}{2} \setminus E(F)$ and $|B_F| \geq \alpha n^2$.
\end{enumerate}
For instance, $\mathcal{F}$ can be taken as the set of all expander, joined, non-Hamiltonian subgraphs on a vertex set $V$, and for each $F \in \mathcal{F}$ we can take $B_F$ as the set of boosters of $F$.

\begin{lemma} \label{lemma:conditionalbooster}
    Let $\alpha,\delta>0$ satisfy $1024 \delta \leq \alpha^2$ and $\delta \leq 1/e$. Let $G = G(n,p)$ for $p \geq 1/n$, and let $(\mathcal{F}, \mathcal{B})$ be a $(\delta p, \alpha)$-boosterable family on $V(G)$. Then, w.h.p. for any $F \in \mathcal{F}$ with $F \subseteq G$, we have $|B_F \cap E(G)| > \alpha p n^2 / 2$.
\end{lemma}

\begin{proof}
    For any $F \in \mathcal{F}$,
    since $B_F \cap E(F) = \emptyset$,
    we have
    \[ \probability( |B_F \cap E(G)| \leq \alpha p n^2 / 2 | F \subseteq G ) = \probability( |B_F \cap E(G)| \leq \alpha p n^2 / 2) \leq e^{- \alpha p n^2 / 8}, \]
    where the last inequality follows from a Chernoff bound, using that $|B_F \cap E(G)|$ is a sum of independent $\{0,1\}$ random variables and $\expectation[|B_F \cap E(G)|] \geq \alpha n^2p$.
    
    Let $q$ be the probability that for some $F \in \mathcal{F}$, $F \subseteq G$ but $|B_F \cap E(G)| \leq \alpha p n^2 / 2$. Then, by a union bound, we have
    \begin{align*}
        q
        & \leq \sum_{F \in \mathcal{F}} \probability( F \subseteq G \land |B_F \cap E(G)| \leq \alpha p n^2 / 2 ) \\
        & = \sum_{F \in \mathcal{F}} \probability( |B_F \cap E(G)| \leq \alpha p n^2 / 2 | F \subseteq G ) \probability(F \subseteq G) \\
        & \leq e^{-\alpha p n^2 / 8} \sum_{F \in \mathcal{F}} \probability(F \subseteq G)
        \leq e^{-\alpha p n^2 / 8} \sum_{F \subseteq K_n, |E(F)| \leq \delta p n^2} \probability(F \subseteq G) \\
        & \leq e^{-\alpha p n^2 / 8 } \sum_{t = 0 }^{\delta p n^2} \binom{n^2}{t} p^t
        \leq  e^{-\alpha p n^2 / 8 } \sum_{t = 0 }^{\delta p n^2} \left( \frac{e n^2 p}{t} \right)^t.
    \end{align*}
    Note that for a fixed $x > 0$, the function $t \mapsto (x/t)^t$ is increasing between $0$ and $x/e$.
    Since $\delta \leq 1/e$,
    we have that each term of the last sum is at most $(e/\delta)^{\delta p n^2}$.
    Therefore we obtain
    \[ q \leq e^{-\alpha p n^2 / 8 } \sum_{t = 0 }^{\delta p n^2} \left( \frac{e n^2 p}{t} \right)^t \leq n^2 e^{-\alpha p n^2 / 8} \left(\frac{e}{\delta}\right)^{\delta p n^2} = n^2 e^{ pn^2( \delta(1 - \ln \delta) - \alpha/8)} \leq n^2 e^{-\alpha p n^2 / 16}, \]
    where in the last step we used $-x \ln x \leq \sqrt{x}$ (valid for all $x > 0$) and $1024 \delta \leq \alpha^2$ to deduce $\delta(1 - \ln \delta) \leq \delta + \sqrt{\delta} \leq 2 \sqrt{\delta} \leq \alpha/16$.
    Then $p \geq 1/n$ shows that the last term is $o(1)$, as required.
\end{proof}

\section{Finding cycles} \label{section:cycles}

After applying the stability results to understand the global structure of the colouring of the random graph $G$, our task will be to find cycles in some monochromatic subgraph related to the extremal colouring.
That subgraph will inherit \emph{expansion} properties of $G$, which we will use to find cycles of prescribed lengths.
The goal of this section is to prove the following lemma.

\begin{lemma}[Cycle finder lemma] \label{lemma:cyclesprescribedlength}
    For each $\lambda > 0$,
    there exist $C, \varepsilon > 0$ such that the following holds with high probability.
    Let $p \geq C/n$ and $G = G(n,p)$.
    For all subgraphs (not necessarily induced) $H\subseteq G$ such that
    \begin{enumerate}
        \item \label{item:cyclesprescribedlength-largevertexset} $|V(H)| > \lambda n$,
        \item
        \label{item:cyclesprescribedlength-mindegree}$\delta(H) \geq \lambda np$, and
        \item \label{item:cyclesprescribedlength-manyedges} $|E(G[V(H)])| - |E(H)| \leq \varepsilon p n^2$;
    \end{enumerate}
    it holds that $C_{\ell} \subseteq H$ for all $\ell \in \{ 3 \log_2 n, \dotsc,  |V(H)|\}$.
\end{lemma}
We also have the following `bipartite' version of \cref{lemma:cyclesprescribedlength}.

\begin{lemma}[Bipartite cycle finder lemma] \label{lemma:cyclesprescribedlength-bipartite}
    For all $\lambda > 0$ there exist $ C, \varepsilon > 0$ such that the following holds with high probability.
    Let $p \geq C/n$ and $G = G(n,p)$.
    For all bipartite subgraphs (not necessarily induced) $H\subseteq G$ with vertex classes $V_1, V_2$ such that
    \begin{enumerate}
        \item \label{item:cyclesprescribedlength-bip-largevertexset} $|V_1| \geq 3|V_2|/2$, $|V_2| > \lambda n$, 
        \item
        \label{item:cyclesprescribedlength-bip-mindegree}$\delta(H) \geq \lambda np$, and
        \item \label{item:cyclesprescribedlength-bip-manyedges} $|G[V_1,V_2]| - |H[V_1,V_2]| \leq \varepsilon p n^2$;
    \end{enumerate}
    it holds that $C_{\ell} \subseteq H$ for all even $\ell \in \{ 3 \log_2 n, \dotsc,  2|V_2|\}$.
\end{lemma}
After gathering some tools on expansion, we will give the proof of \cref{lemma:cyclesprescribedlength} at the end of this section.

\subsection{Expansion from minimum degree}
Now we show that each subgraph of a random graph with sufficiently large minimum degree has nice expansion properties.
Similar arguments in a similar setting were given, e.g. by Krivelevich, Lubetzky and Sudakov~\cite{KLS2014}.

\begin{lemma} \label{lemma:expansionfromminimumdegree}
    For each $\eta, d > 0$, there exist $C, c > 0$ such that the following holds.
    Let $p \geq C/n$ and $G = G(n,p)$. Then w.h.p.
    \begin{enumerate}
        \item for each $G' \subseteq G$ with $\delta(G') \geq \eta pn$, $G'$ is a $(cn,d)$-expander, and
        \item for each bipartite $G' \subseteq G$ with $\delta(G') \geq \eta pn$, $G'$ is a $(cn,d)$-bipartite-expander.
    \end{enumerate}
\end{lemma}

\begin{proof}
    We will only prove the statement for non-bipartite graphs, since the statement for bipartite graphs has essentially the same proof.
    Without loss of generality we assume $d \geq 4$, since that gives a stronger statement.
    Set
    \begin{equation*}
        c = \eta^2/(16 e^2 (2d+1)^2)\quad \text{ and }\quad
        C = \max \left\{ 4 (d+1) \eta^{-1}, 8 \eta^{-1} \right\}.
    \end{equation*}
    
    Let $k = c n$.
    Suppose $G' \subseteq G$ is a graph with $\delta(G') \geq \eta p n$ which is not a $(k, d)$-expander.
    Let $X \subseteq V(G')$ such that $|X| \leq k$ and $|N_{G'}(X) \setminus X| \leq d |X|$, and let $Y = N_{G'}(X) \setminus X$.
    As $\delta(G') \geq \eta p n$, we conclude that from the $\binom{|X|}{2} + |X||Y|$ potential edges incident to $X$ in $G'$, at least $\eta p n |X|/2$ of them are actually present in $G'$, and thus also in $G$.
    The probability of that event can be bounded from above by using a union bound over all possible choices of $X$ and $Y$,
    and we obtain (with explanations to follow)
    \begin{align*}
        \sum_{1 \leq x \leq k} \binom{n}{x} \binom{n}{dx} \binom{x^2/2 + dx^2}{\eta p n x/2} p^{\eta p n x/2}
        & \leq \sum_{1 \leq x \leq k} \left(\frac{en}{x}\right)^x \left(\frac{en}{dx}\right)^{dx} \left(\frac{e(x^2+2dx^2)}{\eta p n x}\right)^{\eta p n x/2} p^{\eta p n x/2} \\
        & = \sum_{1 \leq x \leq k} \left[ \frac{e^{d+1}}{d^d} \left( \frac{n}{x} \right)^{d+1} \left( \frac{x}{n} \right)^{\eta p n /2}  \left( \frac{e(2d+1)}{\eta} \right)^{\eta p n /2} \right]^x \\
        & \leq \sum_{1 \leq x \leq k} \left[ \frac{e^{d+1}}{d^d} \left( \frac{x}{n} \right)^{\eta p n / 4}  \left( \frac{e(2d+1)}{\eta} \right)^{\eta p n /2} \right]^x \\
         & \leq \sum_{1 \leq x \leq k} \left[ \left( \frac{x}{n} \right)^{\eta p n / 4}  \left( \frac{e(2d+1)}{\eta} \right)^{\eta p n /2} \right]^x \\
        & = \sum_{1 \leq x \leq k} \left[ \left( \frac{x}{n} \right)^{1 / 4}  \left( \frac{e(2d+1)}{\eta} \right)^{1/2} \right]^{\eta p n x}.
    \end{align*}
    Here, in the first inequality we used the bound $\binom{n}{k} \leq (en/k)^k$; in the second inequality we used the choice of $C$ and $p \geq C/n$ to deduce $d+1 \leq \eta p n / 4$, and in the third inequality we used that $d \geq 4$ to deduce $e^{d+1}/d^d \leq 1$.
    
    We separate the last sum in two terms $S_1$ and $S_2$, where $S_1$ consists on the sum of the first $k_1 := (c/p)^{1/2}$ terms, and $S_2$ consists of the sum of the terms from $k_1$ to $k$.
    We will show that both $S_1$ and $S_2$ are $o(1)$,
    from which we conclude that the lemma holds.
    
    We begin by considering $S_1$.
    Since $p \geq C/n$ we have $k_1 = O(n^{1/2})$ and therefore for each $x \leq k_1$ we have $x/n = O(n^{-1/2})$.
    Thus we have $(x/n)^{1/4} = O(n^{-1/8})$, and therefore
    \begin{align*}
        S_1
        & = \sum_{1 \leq x \leq k_1} \left[ \left( \frac{x}{n} \right)^{1 / 4}  \left( \frac{e(2d+1)}{\eta} \right)^{1/2} \right]^{\eta p n x} 
        \leq \sum_{1 \leq x \leq k_1} \left( O(n^{-1/8}) \right)^{\eta p n x} \\
        & \leq \sum_{1 \leq x \leq k_1} \left[ O(n^{-\eta p n /8}) \right]^x 
        \leq  \sum_{1 \leq x \leq k_1} \left[ O(n^{-\eta C /8}) \right]^x,
    \end{align*}
    where the last inequality follows since $p \geq C/n$, which we use to establish $O(n^{- \eta p n/8}) = O(n^{- \eta C /8})$.
    For large enough $n$, the term $O(n^{- \eta C/8})$ is less than $1$, and therefore
    \begin{align*}
        S_1
        & \leq \sum_{1 \leq x \leq k_1} \left[ O(n^{-\eta C /8}) \right]^x
        \leq \sum_{1 \leq x \leq k_1} O(n^{-\eta C /8}) = O(k_1 n^{- \eta C /8})=O(n^{1/2 - \eta C/8}).
    \end{align*}
    Since $C \geq 8 \eta^{-1}$, we conclude that $S_1 = O(n^{-1/2}) = o(1)$.
    
    Now we consider $S_2$.
    For each $x \leq k = cn$ we have $x/n \leq c$.
    Hence, we have
    \begin{align*}
        S_2
        & = \sum_{k_1 \leq x \leq k}  \left[ \left( \frac{x}{n} \right)^{1 / 4}  \left( \frac{e(2d+1)}{\eta} \right)^{1/2} \right]^{\eta p n x} \\
        & \leq \sum_{k_1 \leq x \leq k} \left[ c^{1 / 4}  \left( \frac{e(2d+1)}{\eta} \right)^{1/2} \right]^{\eta p n x}
        \leq \sum_{k_1 \leq x \leq k} \left( \frac{1}{2} \right)^{\eta p n x},
    \end{align*}
    where in the last inequality we used $c \leq \eta^2/(16 e^2 (2d+1)^2)$.
    Using the bound $x \geq k_1 = (c/p)^{1/2}$ and $p \geq C/n$, we have that $\eta p n x = \Omega(n^{1/2})$.
    Since $k = O(n)$, we conclude that
    \begin{align*}
        S_2
        & \leq \sum_{k_1 \leq x \leq k} \left( \frac{1}{2} \right)^{\Omega(n^{1/2})} = O(k 2^{-\Omega(n^{1/2})}) = O(n 2^{-\Omega(n^{1/2})}) = o(1),
    \end{align*}
    as required.
\end{proof}

\subsection{Sparsification}
The following lemma allows us to find subgraphs of graphs which still retain certain minimum degree and joinedness properties, but are way sparser.

\begin{lemma}[Sparsificator] \label{lemma:sparsificator}
	Given $\eta, c > 0$ and $0 < \delta \leq 1/4$, there exists $C \geq 0$ such that the following holds.
	Let $G$ be an $n$-vertex graph, let $H$ be a (not necessarily spanning) subgraph of $G$, let $p \geq C/n$, and let $X, Y \subseteq V(G)$ be such that
	\begin{enumerate}
		\item $\delta(H) \geq \eta pn$,
		\item for each pair of disjoint $2cn$-sets of vertices $A \subseteq X$, $B \subseteq Y$, $e_{H}(A,B) \geq p c^2 n^2$, and
		\item $|E(H)| \leq pn^2$. 
	\end{enumerate}
	Then there exists $H' \subseteq H$ with $V(H') = V(H)$ such that
	\begin{enumerate}
		\item $\delta(H') \geq \delta \eta pn $,
		\item for each pair of disjoint $2cn$-sets of vertices $A \subseteq X$, $B \subseteq Y$, $e_{H'}(A,B) > 0$, and
		\item $|E(H')| \leq 4 \delta pn^2$. 
	\end{enumerate}
\end{lemma}

%

\begin{proof}[Proof of \cref{lemma:sparsificator}]
	Assume $1/C \ll \eta, \delta, c$.
	Let $H' \subseteq H$ be a graph formed by selecting each edge of $G$ with probability $2 \delta$.
	We will show that $H'$ satisfies the required properties simultaneously with non-zero probability.
	For this, we state and prove four claims.
	
	\begin{enumerate}[\textbf{(A\arabic*)},topsep=1em]
        \item \label{claim:notmanyedges} $\probability[|E(H')| \leq 4 \delta n^2 p] \geq 1 - e^{-2n}$.
    \end{enumerate}
	
    Indeed, we have $|E(H)| \leq n^2 p$, and thus the expected number of edges in $H'$ is at most $2 \delta |E(H)| \leq 2 \delta n^2 p$.
	Therefore, by Chernoff's inequality, we have that $|E(H')| \leq 4 \delta n^2 p$ happens with probability at least $1 - \exp(- 3 \delta n^2 p / 4 ) \geq 1 - e^{-2n}$, where in the last inequality we used $p \geq C/n$ and that $C$ is large.
	This proves~\ref{claim:notmanyedges}.

    \begin{enumerate}[\textbf{(A\arabic*)},resume,topsep=1em]
        \item \label{claim:stilljoined}
        With probability at least $1 - e^{-2n}$, each pair of disjoint $2cn$-sets of vertices $A \subseteq X$, $B \subseteq Y$, $e_{H'}(A,B) > 0$.
    \end{enumerate}

	Indeed, for a given pair of disjoint $2cn$-sets of vertices $A$ and $B$,
	the probability that $e_{H'}(A,B) = 0$ is at most $(1 - \delta)^{pc^2n^2} \leq e^{-\delta p c^2 n^2}$.
	Taking an union bound over the at most $\binom{n}{2cn}^2$ possible choices of $A$ and $B$, the probability of failure is at most
	\[ \binom{n}{2cn}^2 e^{-\delta p c^2 n^2} \leq n^{4cn} e^{-\delta p c^2 n^2} = e^{4cn  - \delta pc^2 n^2} \leq e^{-2n}, \]
	where in the last inequality we used $p \geq C/n$ and $1/C \ll \delta, c$ to deduce $\delta pc^2n \geq \delta C c^2 \geq 2 + 4c$.
	Thus~\ref{claim:stilljoined} holds.

    \begin{enumerate}[\textbf{(A\arabic*)},resume,topsep=1em]
    \item 
     \label{claim:firstrequisiteOWLLL}
		For each $v \in V(H)$, $\probability[d_{H'}(v) < \delta \eta n p] \leq 1 - e^{-1}$.
	\end{enumerate}	
To see this, let $v \in V(H)$ and $k = d_H(v)$.
	First, we note that
		\begin{align}
			(2 \delta k - \delta \eta n p)^2 & \geq \delta^2 k^2. \label{equation:chernoffvertexinequality}
		\end{align}
	Indeed, after rearranging and simplifying, we see this is equivalent to
	$k \geq \eta n p$, which holds by assumption.
		
	Note that $d_{H'}(v)$ is a Bernoulli random variable with expectation $2 \delta k$.
	Thus
	\begin{align*}
		\probability[d_{H'}(v) < \delta \eta n p]
			& = \probability[d_{H'}(v) - \expectation[d_{H'}(v)] < \delta \eta n p - 2 \delta k],
		\end{align*}
		and using a Chernoff bound we have
		\begin{align*}
		\probability[d_{H'}(v) < \delta \eta n p]
			& \leq \exp\left( -(2 \delta k - \delta \eta n p)^2/(4 \delta k) \right) \leq e^{- \delta k/4},
		\end{align*}
	where we used the inequality \eqref{equation:chernoffvertexinequality} in the last step.
    Next, by assumption and $p \geq C/n$ we have that $k = \deg_H(v) \geq \delta(H) \geq \eta p n \geq \eta C$, hence the last probability is at most $e^{- \delta \eta C / 4}$.
    Since $1/C \ll \eta$ we can assume that the last term is at most $1 - e^{-1}$.
	We conclude~\ref{claim:firstrequisiteOWLLL} is true.
	
	\begin{enumerate}[\textbf{(A\arabic*)},resume,topsep=1em]
        \item \label{claim:applyOWLLL}
		$\probability[\delta(H') \geq \delta \eta n p] \geq e^{-n}$.
	\end{enumerate}	

    To see this, 
    for each $v \in V(H)$, let $A_v$ be the event that $d_{H'}(v) \geq \delta \eta n p$.
    We have that
    \begin{equation*}
        \probability[\delta(H') \geq \delta \eta n p] =
        \probability\left[ \bigcap_{v \in V(H)} A_v \right] \geq 
        \prod_{v \in V(H)} \probability[A_v],
    \end{equation*}
    where the last inequality follows from the FKG inequality, since all the events $A_v$ are monotone increasing.
    Using \ref{claim:firstrequisiteOWLLL} we have $\probability[A_v] \geq e^{-1}$ for each $v \in V(G)$, and therefore
    \begin{equation*}
        \probability[\delta(H') \geq \delta \eta n p] \geq
        \prod_{v \in V(H)} e^{-1} = e^{-n},
    \end{equation*}
    as desired.

	Now, by \ref{claim:notmanyedges}, \ref{claim:stilljoined} and \ref{claim:applyOWLLL}, we have that, with probability at least $e^{-n} - 2e^{- 2n} > 0$, all the required events are true simultaneously.
	Thus there exists $H' \subseteq H$ with the desired characteristics.
\end{proof}

\subsection{Proof of the Cycle finder lemmas}
Now we can give the proof of \cref{lemma:cyclesprescribedlength}.

\begin{proof}
    We begin by defining $C, \varepsilon, c, b, \delta$ satisfying the hierarchies  $1/C \ll \varepsilon \ll c \ll b \ll \delta \ll \lambda,  1$. 
    We collect a series of properties which occur with high probability in $G$.
    By \cref{lemma:upper-uniform} we have that with probability $1 - o(1)$,
    \begin{enumerate}[\textbf{(P\arabic*)},topsep=1em]
        \item \label{assumption:Guniform} $G$ is $(c/4, p)$-uniform, and
        \item \label{assumption:Gsparse} $|E(G)| \leq pn^2$.
    \end{enumerate}
    Apply \cref{lemma:expansionfromminimumdegree} with $\delta \lambda, b, 6$ playing the roles of $\eta, c, d$; this is possible since $1/C, b \ll \delta, \lambda, 1$.
    Then, with probability $1 - o(1)$ we have that
    \begin{enumerate}[\textbf{(P\arabic*)} ,topsep=1em,resume]
        \item \label{assumption:mindegreegivesexpansion} each subgraph $H \subseteq G$ with $\delta(H) \geq \delta \lambda pn$ is a $(bn, 6)$-expander.
    \end{enumerate}
    
    Let $\mathcal{F}$ be the family of all graphs $F$ with $V(F) \subseteq V(G)$ with $|V(F)| \geq \lambda n / 2$, at most $5 \delta p n^2$ edges, which are $(bn, 5)$-expanders, $2cn$-joined, and do not contain a cycle on $V(F)$.
    By \cref{lemma:tapaoenboosters} (applied here with $|V(F)|, bn, 5, 2cn, 2cn$ playing the roles of $n, k, d, m, M$, respectively), we deduce that for each graph $F \in \mathcal{F}$ there is a set $B_F$ of boosters of $F$ of size at least $\frac{1}{16}\binom{|V(F)|}{2} \geq \lambda^2 n^2/66$.
    Then $(\mathcal{F}, \{ B_F\}_{F \in \mathcal{F}} )$ is a $(5\delta p, \lambda^2 / 66)$-boosterable family.
    
    Then, by the choice $\delta \ll \lambda$, we can apply \cref{lemma:conditionalbooster} (with $5 \delta, \lambda^2/66$ playing the roles of $\delta, \alpha$) to deduce that with probability $1 - o(1)$ we have that
    \begin{enumerate}[\textbf{(P\arabic*)},topsep=1em,resume]
        \item \label{assumption:expansiongivesboosters} for each $F \in \mathcal{F}$ with $F \subseteq G$, either $F$ contains a cycle on $v(F)$ vertices or there are at least $\lambda^2 p n^2 / 132$ boosters for $F$ in $G[V(F)]$.
    \end{enumerate}
    From now on, we condition on \ref{assumption:Guniform}--\ref{assumption:expansiongivesboosters} being true (which happen simultaneously with probability $1 - o(1)$), and we show that the desired properties hold.
    
    Now, let $H \subseteq G$ satisfy \ref{item:cyclesprescribedlength-largevertexset}--\ref{item:cyclesprescribedlength-manyedges} from the statement. By \ref{assumption:Guniform}, we get that for every two sets $X, Y$ of size $2cn$ each, there are at least $(1 - c/4)p(2cn)^2 \geq 2 c^2 p n^2$ edges between them in $G$.
    Since $\varepsilon \ll c$, then~\ref{item:cyclesprescribedlength-manyedges} implies that for every two sets $A, B \subseteq V(H)$ of size $2cn$ each, \[ e_{H}(A,B) \geq e_G(A,B) - \varepsilon p n^2 \geq 2 c^2 p n^2 - \varepsilon p n^2 > c^2 n^2 p.\]
    In particular, $H$ is $2cn$-joined.
    Also, since $\delta(H) \geq \lambda n p \geq \lambda \delta p n$, using~\ref{assumption:mindegreegivesexpansion} we deduce that $H$ is also a $(bn,6)$-expander.
    
    We need to find $C_\ell$ as a subgraph of $H$, for each $3\log_2 n \leq \ell \leq |V(H)|$.
    First, we will use `double brooms' (Definition~\ref{definition:doublebroom}) to find `short' cycles, meaning of length $\ell$ between $3 \log_2 n$ and $|V(H)| - 50cn$.
    Fix such a value of $\ell$.
    Let $h = \lceil \log_2(2cn) \rceil$.
    Since $2h+1 = 2\lceil \log_2(2cn) \rceil + 1 \leq 2 \log_2(2cn) + 3 \leq 3 \log_2n \leq \ell$, we have that $\ell - 2h > 0$.
    Let $T_\ell$ be a $(h, \ell - 2h + 1)$-double-broom, that is, $T_h$ is a tree consisting of two binary trees of depth $h$ whose roots are joined by a path with exactly $\ell - 2h$ edges.
    Note that $\Delta(T_\ell) \leq 3$, $|V(T_\ell)| = 2(2^{h+1} - 1)+ \ell - 2h - 1$, each end set of $T_\ell$ has $2^{h} \geq 2cn$ leaves; and that leaves of $T_\ell$ which are in distinct end sets have distance exactly $\ell$ in $T_h$.
    If $T_\ell \subseteq H$, then since $H$ is $2cn$-joined, we can find an edge between the two sets of leaves of $T_\ell$; and this implies that $C_{\ell} \subseteq H$, as desired.
    We will use Lemma~\ref{lemma:embedding:expander} to argue that $T_\ell \subseteq H$.
    Applying the lemma with $2cn, |V(H)|, 6$ playing the roles of $m, n, d$; we deduce that $T_\ell \subseteq H$ as long as the inequality $|V(T_\ell)| \leq |V(H)| - 30cn$ holds.
    We have that $2^{h+2} \leq 2^{\log_2(2cn) + 3} \leq 16cn$, and therefore
    \begin{align*}
        |V(T_\ell)|
        & = 2(2^{h+1} - 1)+ \ell - 2h - 1 \leq 16cn + \ell
        \leq |V(H)| - 30cn,
    \end{align*}
    as required.
    
    From now on, we will assume that $\ell$ is fixed and $|V(H)| - 50cn <  \ell \leq |V(H)|$ holds; and our task is to prove that $C_\ell \subseteq H$.
    We wish to apply the Sparsificator (\cref{lemma:sparsificator}) with the parameters $\lambda, c, \delta, H, V(H), V(H)$ playing the roles of $\eta, c, \delta, H, X, Y$, respectively.
    Let us check that the hypothesis of that lemma hold with our choice of parameters.
    Indeed, by assumption, we have $1/C \ll \lambda, c, \delta$.
    We also have $\delta(H) \geq \lambda p n$ from \ref{item:cyclesprescribedlength-largevertexset}.
    We already argued that for any two $2cn$-sized sets $A,B \subseteq V(H)$, the bound $e_H(A,B) > c^2 n^2 p$ holds.
    Finally, we have $|E(H)| \leq |E(G)| \leq pn^2$ by \ref{assumption:Gsparse}.
    Hence, the lemma gives $H' \subseteq H$ with $V(H') = V(H)$ such that
    \begin{align}
        \delta(H') \geq \delta \lambda n p, \label{equation:G''mindegree} \\
        \text{$H'$ is $2cn$-joined, and} \label{equation:G''joined} \\
        |E(H')| \leq 4 \delta n^2 p. \label{equation:G''sparse}
    \end{align}
    Using \eqref{equation:G''mindegree} and \ref{assumption:mindegreegivesexpansion}, we deduce
    \begin{equation}
        \text{$H'$ is a $(bn, 6)$-expander.} \label{equation:G''expander}
    \end{equation}
    Let $r = |V(H)| - \ell$.
    Since $\ell \geq |V(H)| - 50 cn$ and $c\ll b$,
    we get that $r \leq bn/10$.
    Then \eqref{equation:G''expander} allows us to apply \cref{lemma:friedmanpippenger-expansion}, and we deduce that there exists a set $X \subseteq V(H)$ of size precisely $r$ such that, setting $H^\ast := H' - X$,
    \begin{equation}
        \text{$H^\ast$ is a $(bn, 5)$-expander on $\ell$ vertices.} \label{equation:Gastexpander}
    \end{equation}
    Since $H^\ast$ is an induced subgraph of $H'$, from \eqref{equation:G''joined}--\eqref{equation:G''sparse} we get that
    \begin{align}
        \text{$H^\ast$ is $2cn$-joined, and} \label{equation:Gastjoined} \\
        |E(H^\ast)| \leq 4 \delta n^2 p. \label{equation:Gastsparse}
    \end{align}
We now define a sequence $H_0 \subseteq H_1 \subseteq H_2 \subseteq \dotsb \subseteq H_n$ such that, for each $i \geq 0$,
    \begin{enumerate}[\textbf{\upshape{(B\arabic*)}}]
        \item \label{item:usingboosters-containment} $H^\ast \subseteq H_i \subseteq H - X$,
        \item \label{item:usingboosters-sparse} $|E(H_i)| \leq 4 \delta n^2 p + i$,
        \item \label{item:usingboosters-hamilton} if $i > 0$, either $H_i$ is Hamiltonian, or $H_i$ has a longer path than  $H_{i-1}$.
    \end{enumerate}
    We start with $H_0 := H^\ast$, which clearly satisfies \ref{item:usingboosters-containment}--\ref{item:usingboosters-hamilton} for $i = 0$.
    Suppose now that $i \geq 1$ and that we have defined $H_0 \subseteq \dotsb \subseteq H_{i-1}$ satisfying \ref{item:usingboosters-containment}--\ref{item:usingboosters-hamilton}.
    If $H_{i-1}$ is Hamiltonian, then setting $H_i := H_{i-1}$ satisfies \ref{item:usingboosters-containment}--\ref{item:usingboosters-hamilton} and we are done.
    Otherwise, we can assume that $H_{i-1}$ is not Hamiltonian.
    By \ref{item:usingboosters-containment}, together with \eqref{equation:Gastexpander}--\eqref{equation:Gastjoined}, we get that $H_{i-1}$ is a $(bn,5)$-expander and $2cn$-joined.
    From \eqref{equation:Gastsparse} and \ref{item:usingboosters-sparse}, we get that $|E(H_{i-1})| \leq 4 \delta n^2 p + i \leq 4 \delta n^2 p + n \leq 5 \delta n^2 p$.
    We also have $|V(H^\ast)| = \ell \geq |V(H)| - 50cn \geq (\lambda - 50c)n \geq \lambda n/2$.
    Therefore, $H_{i-1} \in \mathcal{F}$.
    Then \ref{assumption:expansiongivesboosters} implies that there are at least $\lambda^2 p n^2 / 66$ boosters for $H_{i-1}$ in $G[V(H_{i-1})]$.
    Since $\varepsilon \ll \lambda$; by \ref{item:cyclesprescribedlength-manyedges}, at least one of these edges, say $e$, actually belongs to $H[V(H^\ast)]$,
    and thus to $H - X$.
    Then, the choice $H_i := H_{i-1} + e$ satisfies \ref{item:usingboosters-containment}--\ref{item:usingboosters-hamilton} by the definition of booster.
    
    Since a path in $H - X$ cannot have length larger than $n$, we must have that $H_n$ is Hamiltonian, meaning that $H$ has a cycle of length $|V(H_n)| = |V(H^\ast)| = \ell$, as required.
\end{proof}

It remains to prove the Bipartite cycle finder lemma, \cref{lemma:cyclesprescribedlength-bipartite}.
The proof is essentially the same as the proof of Lemma~\ref{lemma:cyclesprescribedlength}, so we shall only comment on the proof briefly.

\begin{proof}[Proof of \cref{lemma:cyclesprescribedlength-bipartite} (sketch)]
    We follow essentially the same steps which were used in the proof of Lemma~\ref{lemma:cyclesprescribedlength}, with the corresponding changes for bipartite graphs which we now explain.
    We still keep the properties \ref{assumption:Guniform} and \ref{assumption:Gsparse}, but we replace \ref{assumption:mindegreegivesexpansion} and \ref{assumption:expansiongivesboosters} for their bipartite counterparts. By \cref{lemma:expansionfromminimumdegree} we have with high probability that
     \begin{enumerate}[\textbf{(P\arabic*)} ,topsep=1em]
       \setcounter{enumi}{2}
        \item \label{assumption:mindegreegivesbipexpansion} each bipartite subgraph $H$ with $\delta(H) \geq \delta \lambda pn$ is a $(bn, 6)$-bipartite-expander.
    \end{enumerate}    
    Moreover, we define the family $\mathcal{F}$ of all bipartite graphs with $V(F)\subset V(G)$, on classes $V_1, V_2$ with $|V_1| \geq |V_2|$, with at most $5 \delta p n^2$ edges, which are $(bn, 4)$-expanders, $(bn, n/12)$-bipartite-joined. By Lemma \ref{lemma:bipartite-boosters} each of the graphs in $\mathcal{F}$ has at least $n^2/8$ bipartite boosters. Therefore applying \cref{lemma:conditionalbooster} we deduce that the following property holds with high probability
    \begin{enumerate}[\textbf{(P\arabic*)} ,topsep=1em]
       \setcounter{enumi}{3}
        \item \label{assumption:expansiongivesbiboosters} for each $F\in \mathcal{F}$ with $F\subset G$, either $F$ contains a cycle on $2|V_2|$ vertices or there there at keast $pn^2/16$ bipartite boosters for $F$ in the biparte subgraph of $G$ induced by $V_1\times V_2$.
    \end{enumerate} 

    Now we are given a bipartite $H \subseteq G$ as in the statement, with parts $V_1, V_2$.
    The minimum degree property implies that $H$ is a $(bn, 6)$-bipartite-expander.
    Again, we can use double brooms to find all cycles of even length between $3 \log_2 n$ and $2|V_2| - 50 c n$ (the double brooms can be found using Corollary~\ref{corollary:extendable:bipartite} with $S = \emptyset$, see also the proof of Lemma~\ref{lemma:connectingpath} for more details).
    We have therefore reduced the problem to find all cycles of even length $2\ell$ which are between $2|V_2| - 50 cn$ and $2|V_2|$.
    First, we sparsify $H$ using \cref{lemma:sparsificator}, obtaining $H'$.
    With this, deduce that $H'$ is bipartite expander, and using \cref{lemma:friedmanpippenger-expansion-bipartite} we can remove $|V_2|-\ell$ vertices from $V_2$ to pass to $H^\ast$ which is still a bipartite expander, with vertex classes $V_1, V'_2$, and $|V'_2| = \ell$.
    Again, we can argue that $H^\ast$ is a bipartite expander, suitably bipartite-joined, and is sparse, and we can do the booster-adding procedure as before.
\end{proof}

\section{Proof of the main theorem} \label{section:mainproof}

As sketched in the beginning, the key step in the proof of both cases, odd and even, consists of finding a more refined version of the monochromatic subgraphs given by our Stability lemma (Theorem~\ref{theorem:stability:1}); and then applying the Cycle finder lemmas (Lemma~\ref{lemma:cyclesprescribedlength} and Lemma~\ref{lemma:cyclesprescribedlength-bipartite}) directly in those monochromatic subgraphs.
This `refining' step corresponds to Step 2 outlined in Section~\ref{section:proofsketch}.

We now sketch this refinement step. To prove this, we will start from a partition $V'_1, V'_2$, as given by the Stability lemma, and then we will reallocate some vertices carefully.
This will be done first by removing vertices having only a few neighbours in one of $V'_1, V'_2$; and then by reallocating vertices if they have many neighbours in the `wrong colour' in one of the $V'_1, V'_2$.
The key observation here is that at the end of this cleaning, no more than $O(1/p)$ vertices are removed from $V'_1, V'_2$, which is shown by tracking carefully how the vertices are removed.

The next lemma will be used in this `fine stability' analysis.
It will allow us to find a long path between any pair of vertices in different parts of an expander bipartite graph.
The proof follows from tree-embedding considerations in bipartite expanders, and is given for completeness in Appendix~\ref{section:appendix-proofs}.

\begin{lemma}\label{lemma:connectingpath}Let $n,m,d\in\mathbb N$ satisfy $n\ge 6dm$, and $d\ge 7$, and suppose that $n$ is odd.
Let $G$ be a bipartite graph with parts $V_1$ and $V_2$ such that $|V_1|,|V_2|\ge 3n/2$,
and suppose $G$ is $(m,d)$-bipartite-expander and m-joined.
Then, for every $s_1\in V_1$, $s_2\in V_2$, there exists an $(s_1,s_2)$-path in $G$ of length $n-2$. 
\end{lemma}

Now we are ready to prove our main result.

\begin{proof}[Proof of Theorem~\ref{theorem:main}]

Since part~\ref{item:bestpossible} has already been proven, we only need to show~\ref{item:maincycles}.
We begin by choosing constants 
\[1/n\ll 1/C\ll\eta\ll\delta\ll\beta\ll\alpha\ll c \ll 1\]
so that Theorem~\ref{theorem:stability:1} holds for $\beta/4$, $\delta$ and $C$ playing the role of $K$, and Lemmas~\ref{lemma:upper-uniform},~\ref{lemma:upper-small-uniform},~\ref{lemma:gnp:1},~\ref{lemma:cyclesprescribedlength}, ~\ref{lemma:cyclesprescribedlength-bipartite},~\ref{lemma:expansionfromminimumdegree} hold.
Let $N= R(C_n)+C/p$ and let $G=G(N,p)$ with $p\geq C/N$. Each of the following properties holds with high probability.
\begin{enumerate}[\textbf{(G\arabic*)}]
    \item\label{Gnp:1}  $G$ is $(\eta,p)$-uniform,
    \item\label{Gnp:1s}  $G$ is $(\eta,p)$-upper-small-uniform,
    \item\label{Gnp:2} every subgraph $G'\subset G$ with $\delta(G')\ge pn/100$ is a $(cn,10)$-expander, 
    \item\label{Gnp:3} every bipartite subgraph $G'\subset G$ with $\delta(G')\ge pn/100$ is a $(cn,10)$-bipartite-expander, and
    \item\label{Gnp:4} for every pair of disjoint subsets $X,U\subset V(G)$, with $|X|\ge 1000/p$ and $|U|\ge n/4$, we have 
    \[\frac{p}{2}|U||X|\le e(U,X)\le 2p|U||X|.\]
\end{enumerate}

By conditioning on those events, we can assume that \ref{Gnp:1}--\ref{Gnp:4} hold for $G$ from now on.
We separate the proof in two cases depending on the parity of $n$.

\medskip \noindent\emph{Case 1: $n$ is odd.}
In this case, $R(C_n) = 2n-1$ and thus we can assume $N = 2n + C/p$.
Suppose that there exists a $2$-colouring  $c:E(G)\to \{R,B\}$ such that $G$ has no monochromatic cycle of length $n$.
Let $G_R$ and $G_B$ denote the red graph and blue graph, respectively.
By modifying the partition given by the Stability lemma, we obtain the following `medium-fine' stability structure.

\begin{claim}[Fine stability, odd case]\label{claim:stability:odd}
There exists a partition $V(G)=V_1\cup V_2\cup W$ such that 
\begin{enumerate}[\textbf{\upshape{(S\arabic*)}}]
    \item\label{stability:1} $|V_1|,|V_2|\ge (1-\alpha)n$,
    \item\label{stability:3} $e(G_R[V_i])\ge (1-\alpha)p\tbinom{|V_i|}{2}$ for $i\in \{1,2\}$,
    \item\label{stability:4} $e(G_B[V_1,V_2])\ge (1-\alpha)p|V_1||V_2|$, and
    \item\label{stability:2} for $i\in \{1,2\}$, every vertex $v\in V_i$ satisfies $d_R(v,V_i)\ge pn/100$ and $d_B(v,V_{3-i})\ge pn/100$.
\end{enumerate}
\end{claim}

\begin{proofclaim}
As a first step, we use Theorem~\ref{theorem:stability:1} (with $\beta/4$ playing the role of $\alpha$) to find a partition $V(G)=V'_1\cup V'_2$ with the following properties: 
\begin{enumerate}
    \item $|V'_1|,|V'_2|\ge (1-\beta/4)n \geq (1-\beta)n$;
    \item for $i\in\{1,2\}$, $e(G_R[V'_i])\ge (1-\beta/4)(1-\eta)p\binom{|V' _i|}{2} \geq (1-\beta)p\binom{|V' _i|}{2}$; and
    \item $e(G_B[V'_1,V'_2])\geq (1-\beta/4)(1-\eta)p|V'_1||V'_2| \ge (1-\beta)p|V'_1||V'_2|$.
\end{enumerate}
In the inequalities above we have used \ref{Gnp:1} and that $\eta\ll \beta$.

For $i\in\{1,2\}$, let $X_i$ be the set of vertices in $x\in V'_i$ such that $d_B(v,V_i')\ge pn/100$,
and let $Y_i$ be the set of vertices $x\in V_i'$ such that $d_R(v,V_{3-i}')\ge pn/100$.
Thus $X_i \cup Y_i$ are the vertices in $V'_i$ adjacent to many neighbours in the `wrong colour'.

For each $i\in\{1,2\}$, we have $|V'_i| \leq N\le 2n + C/p\le 3n$, and, since $G$ is $(\eta,p)$-uniform, $e(G[V'_i]) \leq (1 + \eta)p\binom{|V'_i|}{2}$.
Hence, we have
\begin{align*}
    \frac{pn |X_i|}{100}
    & \leq \sum_{x \in X_i} d_B(v, V_i) \leq 2 e(G_B[V'_i]) = 2 (e(G[V'_i]) - e_R(G[V'_i])) \\
    &\textstyle\le 2((1+\eta)p\binom{|V'_i|}{2}-(1-\beta)p\binom{|V'_i|}{2})
     = \textstyle 2p\binom{|V'_i|}{2}(\beta + \eta)
    \leq 18 \beta p n^2,
\end{align*}
where in the last line we used $\eta \ll \beta$ and $|V'_i|\le 3n$.
We deduce that $|X_1|, |X_2| \le 1800\beta n$.
Very similar calculations can be used to get that $|Y_i|\le 1800\beta n$.

Let $V'_{1,1}=V_1'\setminus (X_1\cup Y_1)$ and $V'_{2,1}=V_2'\setminus(X_2\cup Y_2)$.
For $i\ge 1$, if there exists a vertex $v_i\in V'_{1,i}\cup V'_{2,i}$ such that $d(v_i,V_{1,i})\le pn/25$ or $d(v_i,V_{2,i})\le pn/25$,
then we update $V_{1,i+1}=V_{1,i}\setminus \{v_i\}$ and $V_{2,i+1}=V_{2,i}\setminus \{v_i\}$.
If no such vertex exists, then we stop this process.
We claim that this process stops at some time $j^*$ such that $1\le j^*\le 2000/p$.
Otherwise, suppose $j^*\ge \lceil \frac {2000}p\rceil=l$,
let $Z=\{v_1,\dots, v_{l}\}$ denote the set of vertices we have removed up to step $l$, and let $Z'\subset Z$ be the set of those vertices having less than $pn/25$ neighbours in $V_{1,l}$.
If $|Z'|\ge 1000/p$, then by property~\ref{Gnp:4} we have 
\[\frac{pn|Z'|}{25}\ge e(Z',V_{1,l})\ge \frac{p|Z'||V_{1,l}|}{2}\ge \frac{pn|Z'|}8,\]
where in the last inequality we used $\beta, 1/C \ll 1$ to deduce that $|V'_{1,l}| = |V'_1| - |X_1| - |Y_1| - l \geq (1 - \beta)n - 3600\beta n - 2000n/C - 1 \geq n/2$.
This is a contradiction, so we deduce that $|Z'| < 1000/p$ and therefore $|Z\setminus Z'| > 1000/p$.
Proceeding similarly, again by~\ref{Gnp:4} we have
\[\frac{pn|Z\setminus Z'|}{25} \geq e(Z\setminus Z',V_{2,l})\ge \frac{p|Z\setminus Z'||V_{2,l}|}{2}\ge \frac{pn|Z\setminus Z'|}{4},\]
where the bound in $|V_{2,l}|$ is obtained as the one for $|V_{1,l}|$.
Again, this is a contradiction.
We deduce that indeed the process stopped in some step $j^\ast \leq 2000/p$. 

Let $V_1=V_{1,j^*}$, $V_2=V_{2,j^*}$, and $W=V(G)\setminus (V_1\cup V_2)$.
We will check that propierties~\ref{stability:1}--\ref{stability:4} hold.
Since $j^*\le 2000/p \leq 2000n/C$ and $|X_i\cup Y_i|\le 3600\beta n$, we have that
$|V'_i \setminus V_i| = |X_i \cup Y_i| + j^\ast \leq 3600 \beta n$, where we used $1/C \ll \beta$ in the last step.
Since $\beta \ll \alpha$, this inequality together with $|V'_i| \geq (1 - \beta)n$, imply~\ref{stability:1}.

To see that~\ref{stability:3} holds, consider the set $V'_i \setminus V_i$.
We know already that $|V'_i \setminus V_i| \leq 3600 \beta n$ holds.
Next, using \ref{Gnp:1} or \ref{Gnp:1s} (depending if $|V'_i \setminus V_i| \geq \eta n$ or not) we get, in any case, that 
$e(G[V'_i \setminus V_i]) \leq 7 p (3600 \beta n)^2$ and $e(G[V'_i \setminus V_i,V_i]) \leq 7 p (3600 \beta n) |V_i|$.
Using all of this, and the fact that $\beta \ll \alpha$, we get
\begin{eqnarray*}e(G_R[V_i])&\ge& e(G_R[V'_i])-e(G[V'_i\setminus V_i])-e(G[V'_i\setminus V_i,V_i])\\
&\ge &(1-\beta)\tbinom{|V_i'|}{2}-7 p (3600 \beta n)^2-7 p (3600 \beta n) |V_i|\\
&\ge&(1-\alpha)p\tbinom{|V_i|}{2}.\end{eqnarray*}
Property~\ref{stability:4} follows from a similar calculation.

We now check the degree conditions. For $i\in \{1,2\}$ and $v\in V_i$, by definition we have 
\[d_R(v,V_i)\ge d(v,V_i)-\frac{pn}{100}\ge \frac{pn}{25}-\frac{pn}{100}\ge \frac{pn}{100}\]
and
\[d_B(v,V_{3-i})\ge d(v,V_{3-i})-\frac{pn}{100}\ge \frac{pn}{25}-\frac{pn}{100}\ge \frac{pn}{100}, \]
proving~\ref{stability:2}.
This finishes the proof of the claim.
\end{proofclaim}


\begin{claim}
$G_B[V_1,V_2]$ is $(cn,10)$-bipartite-expander and $5 \sqrt{\alpha}n$-bipartite-joined.
\end{claim}

\begin{proofclaim}
The first property follows from Lemma~\ref{lemma:expansionfromminimumdegree} (applied with $\eta = 1/100$ and $d=10$) as $\delta(G_B[V_1,V_2])\ge pn/100$.
Now, suppose that $G_B[V_1,V_2]$ is not $5 \sqrt{\alpha}n$-bipartite-joined, and let $X\subset V_1$ and $Y\subset V_2$ be subsets of size $5 \sqrt{\alpha}n$ with no blue edges in between. Since $G$ is $(\eta,p)$-uniform, we have
\[e(G_R[X,Y])\ge (1-\eta)p|X||Y|\ge (1-\eta)25 \alpha pn^2 > 20 \alpha p n^2,\]
where we used $\eta \ll 1$ in the last step.
On the other hand, since we have $|V_1|, |V_2| \leq N \leq 2n+C/p \leq 3n$, we get
\begin{align*}
    e(G_R[V_1,V_2])=e(G[V_1,V_2])-e_B(G[V_1,V_2])
    &\le (1+\eta)p|V_1||V_2|-(1-\alpha)p|V_1||V_2|\\
    &\le 18 \alpha pn^2,
\end{align*}
a contradiction. 
\end{proofclaim}

We now want to use Lemma~\ref{lemma:connectingpath} to refine the partition further and obtain a very fine stability, as sketched before.
Let $B\subset W$ denote the set of vertices having less than $pn/25$ neighbours either in $V_1$ or $V_2$.
As in the proof of Claim~\ref{claim:stability:odd}, using~\ref{Gnp:4} we can deduce that $|B|< 2000/p$.
Now, suppose there is a vertex in $v \in W \setminus B$ with at least $pn/100$ blue neighbours both in $V_1$ and in $V_2$.
Since $G_B[V_1,V_2]$ is $(cn,10)$-bipartite-expander and $5\sqrt{\alpha}n$-bipartite-joined, we can use Lemma~\ref{lemma:connectingpath} to find an $N_B(v,V_1),N_B(v,V_2)$-path of length $n-1$, which together with $w$ completes a cycle of length $n$, a contradiction.
For each $i \in \{1, 2\}$, let $W_i\subset W\setminus B$ be the set of vertices having less than $pn/100$ blue neighbours in $V_{i}$.
We have proved that $W = B \cup W_1 \cup W_2$.
Moreover, for each $i \in \{1, 2\}$ and $v \in W_i$, we have that
\[ d_R(v, V_i) \geq d(v, V_i) - d_B(v, V_i) \geq \frac{pn}{25} - \frac{pn}{100} \geq \frac{pn}{100}. \]

%

We have that $V(G) = B \cup (V_1 \cup W_1) \cup (V_2 \cup W_2)$.
Since $|B| \leq 2000/p$ and $|V(G)| = N \geq 2n+C/p$, we have that one of $V_1 \cup W_1$ or $V_2 \cup W_2$ must have size at least $(N - |B|)/2 \geq n$.
Without loss of generality, assume that it is $V_1 \cup W_1$.
Let $W'_1 \subseteq W_1$ be such that $|V_1 \cup W'_1| = \max\{ n, |V_1| \}$.
Note that by \ref{stability:1} we have that $|W'_1| \leq \alpha n$.

We wish to apply Lemma~\ref{lemma:cyclesprescribedlength} to $H=G_R[V_1 \cup W'_1]$ with $\lambda=1/300$. Note that $|V(H)|=n\geq N/100$, $\delta(H)\geq pn/100 \geq pN/300$. Moreover, the edges missing in $H$ are either the ones in blue in $V_1$, which are at most $\alpha pN^2$ or, or the ones in blue added from adding $W'_1$, which are at most $\alpha  pN^2/100 $. Therefore, we have that $\big| e(G[V(H)])-e(H) \big| \leq 2\alpha p N^2$. If $\alpha$ is small enough, we get that $H$ contains a copy of $C_n$ and we are done. This finishes the proof for odd $n$.

\medskip \noindent\textit{Case 2: $n$ is even.}
Let $n$ be even.
In this case, $R(C_n) = 3n/2 - 1$,
so we can assume that $N=3n/2+C/p$.
Let $G=G(N,p)$.
Suppose that exists a $2$-colouring of $E(G)\to \{R,B\}$ without a monochromatic cycle of length $n$.
In the same way, as in the odd case, we can prove the following refined stability result.

\begin{claim}[Fine stability, even case]\label{claim:stability:even}
There exists a partition $V(G)=V_1\cup V_2\cup W$ such that 
\begin{enumerate}[\upshape{\textbf{(S'\arabic*)}}]
    \item\label{stability:even:1} $|V_1|\ge (1-\alpha)n$.
    \item $|V_2|\ge (1-\alpha)n/2$.
    \item\label{stability:even:3} $e(G_R[V_1])\ge (1-\alpha)p\tbinom{|V_1|}{2}$.
    \item\label{stability:even:4} $e(G_B[V_1,V_2])\ge (1-\alpha)|V_1||V_2|$.
    \item\label{stability:2:even} For $i\in \{1,2\}$, every vertex $v\in V_i$ satisfies $d_R(v,V_i)\ge pn/100$ and $d_B(v,V_{3-i})\ge pn/100$. \hfill $\square$
\end{enumerate} 
\end{claim}
We now refine this structure further.
Let $B\subset W$  be the set of vertices having less than $pn/25$ neighbours in either $V_1$ or $V_2$.
The same argument as before gives that $|B|< 2000/p$.
Let $W_2\subset W\setminus B$ denote the set of vertices having at least $pn/100$ blue neighbours in $V_1$, and let $W_1=W\setminus (B\cup W_2)$ be the set of vertices having less than $pn/100$ blue neighbours in $V_1$.
Let $V_1^*=V_1\cup W_1$ and $V_2^*=V_2\cup W_2$.
Since $N=3n/2+C/p$, then either $|V_1^*|\ge n+C/2p$ or $|V_2^*|\ge n/2+C/2p$ holds.
We split the proof into two cases.

\medskip \noindent\textit{Case 2.1: $|V_1^*|\ge n+C/2p$.}
Using~\ref{stability:2:even} and the definition of $W_1$ we have
\[d_R(v,V_1^*)\ge d_R(v,V_1)\ge pn/25-pn/100\ge pn/100\]
for every $v\in V_1^*$. Moreover, by~\ref{stability:even:3} we have $e(G_R[V_1^*])\ge (1-\alpha)p\tbinom{|V_1|}{2}$.
Therefore, we can use Lemma~\ref{lemma:cyclesprescribedlength} to find a cycle of length $n$ in $G_R[V_1^*]$.

\medskip \noindent\textit{Case 2.2: $|V_2^*|\ge n/2+C/4p$.}
Because of~\ref{stability:2:even} and the definition of $W_2$,
every vertex $v\in V_2^*$ satisfies
\[d_B(v,V_1)\ge pn/25-pn/100\ge pn/100,\]
and every vertex $v\in V_1$ satisfies $d_B(v,V_2^*)\ge d_B(v,V_2)\ge pn/25$.
Then, because of~\ref{stability:2:even}, we can use Lemma~\ref{lemma:cyclesprescribedlength-bipartite} to find a cycle of length $n$ in $G_B[V_1,V_2^*]$.
This finishes all cases.
\end{proof}

\printbibliography

\appendix

\section{Extendability in bipartite graphs} \label{appendix:bipartiteextendibility}
We review tree embeddings in bipartite expander graphs.
For this, we use the notion of `extendability',
as introduced by Glebov, Johanssen and Krivelevich~\cite{Glebov2013} and further developed by Montgomery~\cite{M2019}.

\subsection{Bipartite extendability}
We propose the following natural definition, which is a bipartite version of `extendability' (cf. \cite[Definition 3.1]{M2019}).

\begin{definition}
Let $d\ge 3$ and $m\ge 1$.
Let $G$ be a bipartite graph with parts $V_1$ and $V_2$, and let $S\subset G$ be a subgraph.
We say that $S$ is \textit{$(d,m)$-bipartite-extendable} if $S$ has maximum degree at most $d$ and 
\[|N(U)\setminus V(S)|\ge (d-1)|U|-\sum_{x\in U\cap V(S)}(d_S(x)-1), \]
for all $U\subset V_i$ with $|U|\le 2m$ and $i \in \{ 1,2 \}$.
\end{definition}

The following lemma allows us to extend bipartite-extendable subgraphs by a single edge (cf. \cite[Lemma 5.2.6]{Glebov2013}).

\begin{lemma}\label{lemma:adding:edge:bipartite}
Let $d\ge 3$ and $m\ge 1$. Let $G$ be a bipartite graph with parts $V_1$ and $V_2$, and let $S\subset G$ be a $(d,m)$-bipartite-extendable subgraph of $G$. Suppose that every $i \in \{1,2\}$,
and every subset $U\subset V_i$ with $m\le |U|\le 2m$ satisfies
\[|N(U)|\ge |V(S)\cap V_i|+2dm+1.\]
Then, for every vertex $s\in V(S)$ with $d_S(s)\le d-1$, there exists a vertex $y\in N_G(s)\setminus V(S)$ such that the graph $S+sy$ is $(d,m)$-bipartite-extendable. 
\end{lemma}
\begin{proof}
For $i \in \{1, 2\}$, a subset $X\subset V_i$,
and a subgraph $H\subset G$,
we define
\[D(X;H)=|N(X)\setminus V(H)|-(d-1)|X|+\sum_{x\in X\cap V(H)}(d_H(x)-1).\]
We observe that a subgraph $H\subset G$ is $(d,m)$-bipartite-extendable if for each $i \in \{1,2\}$ and every $X\subset V_i$ with $|X|\leq 2m$ we have $D(X;H)\ge 0$.

\begin{claim}\label{claim:extendability:1} For $i\in \{1,2\}$, suppose that $X\subset V_i$ satisfies $|X|\le 2m$ and $D(X;S)=0$. Then $|X|\le m$. 
\end{claim}

\begin{proofclaim}Suppose that $m<|X|\le 2m$. Then by assumption, we have
\[|N(X)\setminus V(S)|\ge 2dm+1\ge d|X|+1>(d-1)|X|-\sum_{x\in X\cap V(S)}(d_S(x)-1),\]
contradicting that $D(X;S)=0$.
This proves the claim.
\end{proofclaim}

\begin{claim}\label{claim:extendability:2}For $i\in \{1,2\}$, suppose that $X,Y\subset V_i$ satisfy $|X|,|Y|\le 2m$ and $D(X;S)=D(Y;S)=0$. Then $D(X\cup Y;S)=0$.
\end{claim}

\begin{proofclaim}Since $G$ is bipartite, we have that 
\[|N(X\cup Y)\setminus V(S)|=|N(X)\setminus V(S)|+|N(Y)\setminus V(S)|-|N(X\cap Y)\setminus V(S)|.\]
Hence it is easy to check that 
\[D(X\cup Y;S)=D(X;S)+D(Y;S)-D(X\cap Y;S)=-D(X\cap Y;S).\]
 Therefore, as $S$ is $(d,m)$-extendable, we have $D(X\cap Y;S)\ge 0$ and whence $D(X\cup Y;S)\le 0$. However, by Claim~\ref{claim:extendability:1} we deduce that $|X\cup Y|\le 2m$, and then, as $S$ is $(d,m)$-extendable, we have $D(X\cup Y;S)\ge 0$ which implies $D(X\cup Y;S)=0$.
 The claim has been proven.
\end{proofclaim}

For a contradiction, let us suppose that  $s\in V_1$ and that for every $y\in N_G(s)\setminus V(S)$, the graph $S+sy$ is not $(d,m)$-bipartite-extendable.
Therefore, for every $y \in N_G(s) \setminus V(S)$, there exists $i \in \{1, 2\}$ and a set $X_y \subseteq V_i$ with $|X_y| \leq 2m$ such that $D(X_y ; S + sy) < 0$.
Now note that if $X \subseteq V_2$, then since $N(X) \subseteq V_1$ and $N_G(s) \subseteq V_2$,
we have that for any $y \in N_G(s) \setminus V(S)$ it actually holds that $D(X ; S + sy) \geq D(X; S) \geq 0$.
This implies that for every $y \in N_G(s) \setminus V(S)$, we must have $X_y \subseteq V_1$. Therefore, for every $ y\in N_G(s)\setminus V(S)$, there exists a set $X_y\subset V_1$, with $|X_y|\le 2m$, such that $D(X_y;S+sy)<0$.

Note that 
\begin{equation}\label{eq:critical}
    D(X;S+sy)=D(X;S)-\mathbf{1}[y\in N_G(X)]+\mathbf{1}[s\in X]
\end{equation}
holds for all $X\subset V_1$.
Therefore, for all $y \in N_G(s) \setminus V(S)$,
in order to satisfy $D(X_y;S+sy) < 0$,
we must have that $D(X_y;S)=0$, $y\in N(X_y)$, and $s\not\in X_y$.

Let $X^*=\bigcup_{y\in N_G(s)\setminus V(S)}X_y$. Using Claims~\ref{claim:extendability:1} and~\ref{claim:extendability:2}, we deduce that
\begin{enumerate}
    \item\label{eq:extendability:1} $N_G(s)\setminus V(S)\subset N_G(X^*)$, 
    \item\label{eq:extendability:2} $s\not\in X^*$, and
    \item\label{eq:extendability:3} $D(X^*;S)=0$ and $|X^*|\le m$.
\end{enumerate}
Observe that \ref{eq:extendability:1} implies $N(\{s\}\cup X^*\setminus V(S))=N(X^*\setminus V(S))$. Therefore, as $d_S(s)\le d-1$, we have
\begin{equation}\label{eq:extendability:contradiction}D(\{s\}\cup X^*;S)=D(X^*;S)-(d-1)+(d_S(s)-1)\le -(d-1)+(d-2)<0.\end{equation}
However, as $S$ is $(d,m)$-bipartite-extendable and $|\{s\}\cup X^*|\le m+1$, we have $D(\{s\}\cup X^*;S)\ge 0$, which contradicts~\eqref{eq:extendability:contradiction}.
\end{proof}

Now we use Lemma~\ref{lemma:adding:edge:bipartite} iteratively to extend a bipartite-extendable subgraph by attaching a tree 
while retaining bipartite-extendability (cf. \cite[Corollary 3.7]{M2019}).

\begin{corollary}\label{corollary:extendable:bipartite}
Let $d\ge 3$ and $m\ge 1$, and let $G$ be a bipartite graph with parts $V_1$ and $V_2$. Let $S\subseteq G$ be a $(d,m)$-bipartite-extendable subgraph,
and let $T$ be a tree with $\Delta(T) \leq d$ and bipartition classes $D_1$ and $D_2$.
Suppose that $G$ is $m$-bipartite-joined and that, for every $i\in \{1,2\}$, we have \begin{equation}
    |V(S)\cap V_i|+|D_i|+(2d+1)m+1\le |V_i|. \label{equation:assumptionofthetreembeddingcorollary}
\end{equation}
Then, for each $i\in\{1,2\}$, for every $s_i\in S\cap V_i$ with $d_S(s_i)\le d-1$ and $t_i\in D_i$,
there exists a copy $R_i$ of $T$ such that $t_i$ is copied to $s_i$,
$V(R_i) \cap V(S) = \{ s_i \}$,
and $R_i\cup S$ is $(d,m)$-bipartite-extendable. 

\begin{proof}
Let $\{t_i\}=T_0\subset T_1\subset\dots\subset T_l=T$ be a sequence of subtrees, where $T_j$ is obtained from $T_{j-1}$ by adding a leaf. Note that since $G$ is $m$-bipartite-joined, for any $X\subset V_i$ with $|X| \geq m$, the assumption~\eqref{equation:assumptionofthetreembeddingcorollary} implies 
\begin{equation}\label{eq:embedding:extendable}
    N(X)\ge |V_i|-m\ge  |V(S)\cap V_i|+|D_i|+2dm+1. 
\end{equation}
Let $s_i\in V_i$ be a vertex with $d_S(s_i)\le d-1$ and map $t_i$ to $s_i$.
This gives a copy $R^0_i$ of $T_0$,
with $V(R^0_i) \cap V(S) = \{s_i\}$,
and $t_i$ is copied to $s_i$,
and $R^0_i \cup S = S$ is $(d,m)$-bipartite-extendable.

Now, for some $0 \le j < l$, suppose that we have found a copy $R_i^j$ of $T_j$ so that $S\cup R_i^j$ is $(d,m)$-bipartite-extendable, $V(R^j_i) \cap V(S) = \{s_i\}$, and $t_i$ is copied to $s_i$.
Since $\Delta(T)\le d$ and because of~\eqref{eq:embedding:extendable}, we may use Lemma~\ref{lemma:adding:edge:bipartite} to extend the copy of $R_i^j$ to a copy of $R_i^{j+1}$ by adding the leaf that forms $T_{j+1}$, and copying it to a yet unused vertex.
Thus we find a copy $R_i^{j+1}$ of $T_{j+1}$ so that $S\cup R_i^{j+1}$ is $(d,m)$-bipartite-extendable, $V(R^{j+1}_i) \cap V(S) = \{s_i\}$, and $t_i$ is copied to $s_i$.
Then $R_i^l=R_i$ is the desired copy of $T$.    
\end{proof}

\end{corollary}

\subsection{Proof of the lemmas} \label{section:appendix-proofs}
Now we can prove Lemma~\ref{lemma:connectingpath} and Lemma~\ref{lemma:friedmanpippenger-expansion-bipartite}.

\begin{proof}[Proof of Lemma~\ref{lemma:connectingpath}]
Let $I$ denote the graph with vertices $s_1$ and $s_2$ and no edges. An easy case analysis, using that $G$ is a $(m,d)$-bipartite-expander, shows that $I$ is $((d-1)/2,m)$-bipartite-extendable.

Recall that $n$ is odd.
Let $x = (n-3)/2 - \lceil{\log_2m} \rceil$.
We let $T^\ast$ be a broom tree, consisting of a binary tree $B$ of depth $\lceil{\log_2m} \rceil$ whose root vertex is attached to a path $P$ of length $x$.
Let $t \in V(T^\ast)$ be the endpoint of $P$ which is not attached to $B$.
Note that
\begin{enumerate}[\textbf{\upshape{(T\arabic*)}}]
    \item \label{item:vroomvroom-maxdeg} $\Delta(T^\ast) \leq 3 \leq (d-1)/2$,
    \item \label{item:vroomvroom-sizes} $T^\ast$ has bipartition classes $D_1, D_2$ satisfying
    $ |D_1|, |D_2| \leq \left\lceil x/2 \right\rceil + 2m$,
    \item \label{item:vroomvroom-distance} $t$ is at distance $x+\lceil{\log_2m} \rceil = (n-3)/2$ from any other leaf of $T^\ast$,
    \item \label{item:vroomvroom-numberleaves} $T^\ast$ has at least $m$ leaves distinct from $t$.
\end{enumerate} 
We want to use Corollary~\ref{corollary:extendable:bipartite} to find a copy $T_1$ of $T^\ast$ where $t$ is copied to $s_1$, $V(T_1) \cap I = \{s_1\}$ and such that $T_1 \cup I$ is $((d-1)/2, m)$-bipartite-extendable.
We check that the assumptions hold.
We have $I$ is $((d-1)/2, m)$-bipartite-extendable, $d_{I}(s_1) = 0$, and \ref{item:vroomvroom-maxdeg}; so it is only necessary to check that \eqref{equation:assumptionofthetreembeddingcorollary} holds, but this follows from \ref{item:vroomvroom-sizes} since
\[ |V(I) \cap V_i| + |D_i| + dm + 1 \leq 1 + \lceil x/2 \rceil + 2 m + dm + 1 \leq 3n/2 \leq |V_i|. \]
Thus, the desired copy $T_1$ of $T^\ast$ exists.
We repeat the argument, now finding a copy $T_2$ of $T^\ast$ where $t$ is copied to $s_2$ and $V(T_2) \cap (V(T_1) \cup V(I)) = \{ s_2 \}$, the necessary conditions of Corollary~\ref{corollary:extendable:bipartite} are checked by similar calculations.

For each $i \in \{1, 2\}$, let $Q_i \subseteq V(G)$ be the set of leaves in $T_1$ which are not $s_i$.
We must have that $Q_1, Q_2$ are each contained in a different part of the bipartition $\{V_1, V_2\}$.
By \ref{item:vroomvroom-numberleaves}, we have $|Q_1|, |Q_2| \geq m$; and since $G$ is $m$-bipartite-joined there must be an edge $q_1 q_2$ between $q_1 \in Q_1$ and $q_2 \in Q_2$.
By \ref{item:vroomvroom-distance}, for each $i \in \{1, 2\}$ there is a path $P_i$ of length $(n-3)/2$ between $s_i$ and $q_i$.
Their concatenation yields a path of length $2(n-3)/2 + 1 = n-2$ between $s_1$ and $s_2$ in $G$, as required.
\end{proof}

\begin{proof}[Proof of Lemma~\ref{lemma:friedmanpippenger-expansion-bipartite}]
    Let $s \in V_2$ be an arbitrary vertex.
    Note that $S_0 = \{s\}$ is a $(d+2, k)$-bipartite-extendable subgraph of $G$.
    Indeed, $\Delta(S_0) = 0$, and for each $i \in \{1, 2\}$ and $U \subseteq V_i$ with size at most $k$, we have $|N(U) \setminus V(S_0)| \geq |N(U)| - 1 \geq (2d+5)|U| - 1$, where the last inequality follows since $G$ is $(2k,2d+5)$-bipartite-expander.
    Then $|N(U) \setminus V(S_0)| \geq (2d+5)|U| - 1 \geq 
    (d+1)|U|+1 \geq (d+1)|U| - \sum_{x \in U \cap V(S_0)} (d_{S_0}(x) - 1)$, as required.
    
    For each $0 \leq t < 2r$, suppose we have a path $S_t \subseteq G$ of length $t$, starting at $s$, which is a $(d+2, k)$-bipartite-extendable subgraph.
    We construct a path $S_{t+1}$ which extends $S_t$, still has $s$ as an endpoint, and is a $(d+2, k)$-bipartite-extendable subgraph.
    We do this by invoking Lemma~\ref{lemma:adding:edge:bipartite}.
    We need to show that for each $i \in \{1, 2\}$, and every subset $U \subseteq V_i$ with $k \leq |U| \leq 2k$ satisfies $|N(U)| \geq |V(S_t) \cap V_i| + 2(d+2)k + 1$.
    Since $t < 2r$, we have $|V(S_t) \cap V_i| \leq r \leq k$; and together with the fact that $G$ is a $(2k,2 d+5)$-bipartite-expander we indeed have $|N(U)| \geq (2d+5)|U| \geq (2d+5)k \geq |V(S_t) \cap V_i| + 2(d+2)k + 1$, as required.
    Then there exists a vertex $y \notin V(S_t)$ such that $S_{t+1} = S_t + y$ is a path of length $t$ which is a $(d+2,k)$-bipartite-extendable subgraph.
    At the end of this process, we find a path $S_{2r}$ of length $2r$, which starts at $s \in V_2$, which is a $(d+2,k)$-bipartite-extendable subgraph of $G$.
    
    Let $X = V(S_{2r}) \cap V_2$, note that $|X| = r$.
    We claim that $X$ has the required properties, i.e. $G - X$ is a $(2k,d)$-bipartite-expander.
    Indeed, let $i \in \{1, 2\}$ and let $U \subseteq V_i$ have size at most $2k$.
    To conclude, we must show that $|N_{G-X}(U)| \geq d|U|$.
    Using that $S_{2r}$ is a $(d+2, k)$-bipartite-extendable subgraph, we deduce that
    \[ |N_{G-X}(U)| \geq |N_{G}(U) \setminus V(S_{2r})| \geq (d+1)|U| - \sum_{x \in U \cap V(S_{2r})} (d_{S_{2r}(x)} - 1) \geq (d+1)|U| - |U|, \]
    where in the last inequality we used $d_{S_{2r}}(x) \leq 2$, since $S_{2r}$ is a path.
\end{proof}

\end{document}